\documentclass[letterpaper]{amsart}

\usepackage{amssymb}
\usepackage{amsthm}
\usepackage{amsmath}
\usepackage{bm}
\usepackage{amsfonts}
\usepackage{yfonts}
\usepackage{comment}
\usepackage[mathscr]{euscript}
\usepackage{upgreek}
\usepackage[T1]{fontenc}
\usepackage{ wasysym }

\usepackage{stackengine}
\stackMath
\usepackage{tikz-cd}
\usepackage[all]{xy}
\usepackage{xfrac}
\usepackage{MnSymbol}
\usepackage{enumitem}
    
\usepackage{soul}

\usepackage{todonotes}
\usepackage{xargs}

\newcommand{\newdownfree}{\scalebox{1.5}[1]{\ensuremath{\downfree}}}
\newcommand{\newndownfree}{\scalebox{1.5}[1]{\ensuremath{\ndownfree}}}

\newcommandx{\margin}[2][1=]{\todo[linecolor=blue,backgroundcolor=blue!25,bordercolor=blue,#1]{#2}}

\def\Ind#1#2{#1\setbox0=\hbox{$#1x$} \kern\wd0    \hbox to 0pt{\hss$#1\newdownfree$\hss}\kern\wd0}
\def\Notind#1#2{#1\setbox0=\hbox{$#1x$}\kern\wd0 \hbox to 0pt{\hss$#1\newndownfree$\hss}\kern\wd0}

\newcommand{\acl}{\mathrm{acl}}
\newcommand{\dcl}{\mathrm{dcl}}

\newcommand{\cl}{\mathrm{cl}}

\newcommand{\eq}{\mathrm{eq}}

\newcommand{\elesub}{\preccurlyeq}
\newcommand{\res}{\mathord{\upharpoonright}} 

\newcommand{\ff}{\mathbb{F}}

\newcommand{\zz}{\mathbb{Z}}

\newcommand{\ACF}{\mathrm{ACF}}

\newcommand{\sD}{\mathscr{D}}

\newcommand{\sT}{\mathscr{T}}
\newcommand{\sM}{\mathscr{M}}
\newcommand{\sN}{\mathscr{N}}

\newcommand{\sC}{\mathscr{C}}

\newcommand{\Th}{\mathrm{Th}}

\newcommandx{\mult}{\mathrm{mult}}
\newcommand{\monster}{\boldsymbol{\sM}}

\newcommand{\ediag}{\mathrm{Ediag}}

\newcommand{\rs}{\mathrm{rs}}
\newcommand{\core}{\mathrm{es}}

\makeatletter
\newsavebox\myboxA
\newsavebox\myboxB
\newlength\mylenA

\newcommand*\xbar[2][0.75]{%
    \sbox{\myboxA}{$\m@th#2$}%
    \setbox\myboxB\null
    \ht\myboxB=\ht\myboxA%
    \dp\myboxB=\dp\myboxA%
    \wd\myboxB=#1\wd\myboxA
    \sbox\myboxB{$\m@th\overline{\copy\myboxB}$}
    \setlength\mylenA{\the\wd\myboxA}
    \addtolength\mylenA{-\the\wd\myboxB}%
    \ifdim\wd\myboxB<\wd\myboxA%
       \rlap{\hskip 0.5\mylenA\usebox\myboxB}{\usebox\myboxA}%
    \else
        \hskip -0.5\mylenA\rlap{\usebox\myboxA}{\hskip 0.5\mylenA\usebox\myboxB}%
    \fi}
\makeatother

\newtheorem{thm}{Theorem}[section]
\newtheorem{lem}[thm]{Lemma}
\newtheorem*{thm*}{Theorem}
\newtheorem*{thmA*}{Theorem A}
\newtheorem*{prop*}{Proposition}
\newtheorem*{cor*}{Corollary}
\newtheorem*{corB*}{Corollary B}
\newtheorem*{corC*}{Corollary C}
\newtheorem{prop}[thm]{Proposition}
\newtheorem{cor}[thm]{Corollary}
\theoremstyle{definition}

\newtheorem{rem}[thm]{Remark}
\newtheorem{fact}[thm]{Fact}
\newtheorem{example}[thm]{Example}
\newtheorem{question}[thm]{Question}

\definecolor{red}{rgb}{1.0, 0, 0}

\begin{document}
\sloppy

\title[Interpolative fusions] {Interpolative fusions I}
\author{Alex Kruckman, Chieu-Minh Tran, Erik Walsberg}
\email{akruckman@wesleyan.edu, ewalsber@uci.edu, mtran6@nd.edu}
\date{\today}

\begin{abstract}
We define the interpolative fusion $T^*_\cup$ of a family $(T_i)_{i \in I}$ of first-order theories over a common reduct $T_\cap$, a notion that generalizes many examples of random or generic structures in the model-theoretic literature.
When each $T_i$ is model-complete, $T^*_\cup$ coincides with the model companion of $T_\cup = \bigcup_{i \in I} T_i$.
By obtaining sufficient conditions for the existence of $T^*_\cup$, we develop new tools to show that theories of interest have model companions.
\end{abstract}

\maketitle

\section{Introduction}

\noindent It is often desirable to decompose a mathematical structure into simpler components and analyze the structure in terms of how the components behave and interact. In this paper, we take the components to be reducts of the structure, and we are interested in situations when these reducts interact in a definably random fashion modulo some common agreements.
By  Theorem~\ref{thm: first Theorem} below, ``definably random'' in our sense agrees with ``generic'' in the sense of Robinson: e.g., if the first-order theory of each reduct is model-complete, then the original structure satisfies the model companion of the union of these theories.

\medskip\noindent In this paper, we introduce interpolative structures and interpolative fusions as an abstract framework for studying structures and theories that exhibit definably random/generic interactions between certain reducts. We observe that many examples in model theory can be put into this framework,  and we obtain sufficient conditions for first-order logic to be able to capture the aforementioned randomness/genericity. This yields new strategies to show that certain theories have model companions. In subsequent papers, we will  develop more machinery to determine model-theoretic properties of the structure or theory from those of its reducts.

\medskip \noindent Throughout, $I$ is an index set, $L_\cap$ is a first-order language, and $(L_i)_{i \in I}$ is a family of first order languages, all with the same set $S$ of sorts, such that $L_i \cap L_j =L_\cap$ for all distinct $i, j \in I$.
Let $T_i$ be a (possibly incomplete) $L_i$-theory for each $i \in I$,  and assume that each $T_i$ has the same set $T_\cap$ of $L_\cap$-consequences.
(This assumption is quite mild: given an arbitrary family of $L_i$-theories $(T_i)_{i\in I}$, we can extend each $T_i$ with the set of all $L_\cap$-consequences of $\bigcup_{i\in I} T_i$.) 
Set
$$  L_\cup = \bigcup_{i \in I} L_i \quad \text{and} \quad T_\cup = \bigcup_{i \in I} T_i,   $$
and assume that $T_\cup$ is consistent.
(Alternatively, we could just assume that $T_\cap$ is consistent, as these two assumptions are equivalent by Robinson joint consistency; see
Corollary~\ref{cor:interpolation}.)
Finally,  $\sM_\cup$ is an $L_\cup$-structure, $\sM_\square$ is the $L_\square$-reduct of $\sM_\cup$, and $X_\square$ ranges over $\sM_\square$-definable sets (with parameters) for $\square \in I \cup \{ \cap \}$.

\medskip \noindent
Suppose $J \subseteq I$ is finite and $X_i \subseteq M^n$ is $\sM_i$-definable for all $i \in J$.
Then $(X_i)_{i \in J}$ is {\bf separated} if there is a family  $(X^i)_{i \in J}$ of $\sM_\cap$-definable subsets of $M^n$ such that
\[X_i \subseteq X^i \text{ for all } i \in J, \text{and }  \bigcap_{i \in J} X^i =\emptyset .\] 

\noindent We say $\sM_\cup$ is \textbf{interpolative} if for all families $(X_i)_{i\in J}$ such that $J\subseteq I$ is finite and $X_i\subseteq M^n$ is $\sM_i$-definable for all $i\in J$, $(X_i)_{i \in J}$ is separated if and only if $\bigcap_{i\in J} X_i = \emptyset$.
If the class of interpolative models of $T_\cup$ is elementary with theory $T_\cup^*$, then we say that $T_\cup^*$ is the {\bf interpolative fusion} (of $(T_i)_{i \in I}$ over $T_\cap$).
We also say that ``$T^{*}_\cup$ exists'' if the class of interpolative $T_\cup$-models is elementary; this is not automatic, as the definition of an interpolative structure is not first-order.

\medskip \noindent
The name ``interpolative fusion'' comes from the following resemblance with Craig's interpolation theorem:  When $I = \{1,2\}$, $\sM_\cup$ is interpolative if  and only if for all $X_1 \subseteq X_2$, there is an $X_\cap$ such that
$X_1 \subseteq X_\cap$ and $X_\cap \subseteq X_2$.
A natural generalization of Craig's theorem allows us to deduce the following theorem in Section~\ref{sec:fusions}.

\begin{thm} \label{thm: first Theorem}
Suppose each $T_i$ is model-complete. 
Then $\sM_\cup \models T_\cup$ is interpolative if and only if it is existentially closed in the class of  $T_\cup$-models.  
Hence,  $T^*_\cup$ is precisely the model companion of $T_\cup$, if either of these exists.
\end{thm}

\noindent Theorem~\ref{thm: first Theorem} can  be seen as offering a semantic/geometric characterization of the existentially closed $T_\cup$-models and providing a path toward obtaining the model companion of $T_\cup$. For readers who see the notion of interpolative structures as directly reflecting definable randomness and thus fundamental on its own right, Theorem~\ref{thm: first Theorem} can also be read as an explanation for the model-theoretically tame behavior of interpolative structures under favorable conditions. When the $T_i$ are not model-complete, we can still 
view interpolative models as ``relatively existentially closed'' and $T_\cup^*$ as the ``relative model companion'' of $T_\cup$; see Theorem~\ref{thm:relativemc}.

\begin{example}\label{ex:exampleblock1}
The following are natural examples of interpolative fusions. See Section~\ref{sec:examples} for more details. 
\begin{enumerate}
\item The expansion of a theory $T$ by a generic predicate $P$~\cite{Cha-Pi}. This is the interpolative fusion of $T$ with the the theory of an infinite and co-infinite predicate $P$ over the theory of an infinite set. More generally, the model companion of the union of any two model complete theories in disjoint languages, as treated by Winkler in~\cite{Winkler}.

\item Algebraically closed fields with multiple independent valuations \cite{Lou-thesis,Johnson-thesis}. This is the interpolative fusion of multiple copies of $\mathrm{ACVF}$ (with distinct valuation symbols) over $\ACF$. More generally, the model companion of the theory of fields expanded by multiple structures (valuations, derivations, automorphisms, etc.).
\item The group of integers with multiple $p$-adic valuations \cite{AldE}. This is similar to the previous example, using a distinct symbol for each $p$-adic valuation.
\item Algebraically closed fields with generic multiplicative circular orders \cite{Minh-1}. This is the interpolative fusion of $\ACF$ and the theory of circularly ordered multiplicative groups of models of $\ACF$ over the theory of multiplicative groups of models of $\ACF$. If $\mathbb{F}$ is the algebraic closure of a finite field and $\triangleleft$ is any multiplicative circular order on $\mathbb{F}^\times$, then $(\mathbb{F},\triangleleft)$ is a model of this theory. 
\end{enumerate}
\end{example}

\noindent The model companion $T^*$ of a theory $T$ of interest is usually not of the form $T_\cup$ for any nontrivial choice of $(T_i)_{i \in I}$.
However, we can sometimes construct a family $(T_i)_{i \in I}$ of model-complete theories so that $T$ is existentially bi-interpretable with $T_\cup$. One can then deduce that $T_\cup^*$ exists and is existentially bi-interpretable with $T^*$, i.e., $T^*$ is essentially an interpolative fusion. See Section~\ref{sec:interp}
for relevant material, in particular, the precise definition of existential bi-interpretation and the fact that existential bi-interpretations preserve existence of model companions.

\begin{example} \label{ex:exampleblock2}
Each of the following examples is existentially bi-interpretable with an interpolative fusion. See Section~\ref{sec:examples} for details. 
\begin{enumerate}[leftmargin=*]
\item $\mathrm{ACFA}$, the theory of existentially closed difference fields~\cite{Cha-Hru}. More generally, the expansion of any theory by a generic automorphism~\cite{Cha-Pi}.

\item $\mathrm{DCF}_0$, the theory of differentially closed fields of characteristic $0$ ~\cite{RobinsonDCF}. More generally, theories of existentially closed $\sD$-fields in the sense of Moosa and Scanlon~\cite{Moosa-Scanlon}. 

\item The random graph. More generally, the Fra\"iss\'e limits of the classes of finite directed graphs, tournaments, $n$-hypergraphs, $k$-colored graphs etc.

\item Generic Skolemizations, as defined in ~\cite{Winkler} and further studied in~\cite{KRExp}.
\end{enumerate}
\end{example}

\noindent In the examples above, we can deduce existence of $T^*_\cup$ from the fact that $T$ has a model companion with results proven in Sections~\ref{sec:fusions} and~\ref{sec:interp}, but applications of these results can go in the other direction as well. The proof that a theory $T$ of interest has a model companion $T^*$ often involves obtaining a semantic/geometric characterization of the existentially closed models of $T$ and showing that this semantic/geometric characterization is first-order axiomatizable.
In many cases the semantic/geometric characterization is close to the notion of an interpolative structure.
So we may hope to prove the existence of $T_\cup^*$ directly, and then recover the existence of $T^*$.

\medskip \noindent The bulk of the technical work of this paper, Section~\ref{sec: Pseudo-topological base} onward, concerns general machinery that aims to actualize the last statement of the preceding paragraph.
This also provides the reader with the following new strategy to show that a theory $T$ has a model companion $T^*$: 
\begin{enumerate}
    \item Find $(T_i)_{i \in I}$ so that $T_\cup$ is existentially bi-interpretable with $T$;
    \item Use the machinery developed here to conclude that $T_\cup^*$ exists and then, with Corollary~\ref{cor:delta1-cor}, deduce that $T^*$ exists.
\end{enumerate}

\medskip \noindent In each of the examples above, the collection of definable subsets of $T_\cap$-models is equipped with an ordinal-valued dimension satisfying some natural conditions. We refer to this setting as ``pseudo-topological'' and investigate it in Section~\ref{sec: Pseudo-topological base}. We say that an arbitrary set $A$ is {\bf pseudo-dense} in $X_\cap$ if $A$ intersects every $\sM_\cap$-definable $Y_\cap \subseteq X_\cap$ such that $\dim Y_\cap = \dim X_\cap$, and we say that $X_\cap$ is a {\bf pseudo-closure} of $A$ if $A\subseteq X_\cap$ and $A$ is pseudo-dense in $X_\cap$. We say that $\sM_i$ \textbf{has pseudo-closures in $\sM_\cap$} if every $\sM_i$-definable set has a pseudo-closure, and we say $T_i$  \textbf{has pseudo-closures in $T_\cap$} if the same situation holds for every $T_i$-model. We say that $T_i$ {\bf defines pseudo-denseness} over $T_\cap$ if pseudo-denseness is uniformly definable.  

\medskip \noindent We obtain the following general conditions for the existence of interpolative fusions:

\begin{thm} \label{thm:mainsec3}
Suppose  $\dim$ is an ordinal dimension on $T_\cap$, $T_i$ defines pseudo-denseness over $T_\cap$, and $T_i$ has pseudo-closures in $T_\cap$ for all $i\in I$. Then $T^*_\cup$ exists.
\end{thm}

\noindent Without using the condition that $T_i$ defines pseudo-denseness over $T_\cap$ for $i \in I$ in the preceding theorem, we can still show that  $\sM_\cup \models T_\cup$ is interpolative if and only if $\bigcap_{i \in J} X_i \neq \emptyset$ whenever $J \subseteq I$ is finite, $X_i$ is $\sM_i$-definable for all $i \in J$, and there is some  $\sM_\cap$-definable set $X_\cap$ such that each $X_i$ is pseudo-dense in $X_\cap$. 
This property is first-order axiomatizable when pseudo-denseness is definable, and it yields a natural system of ``pseudo-topological'' axioms for $T^{*}_\cup$ .
The pseudo-topological axioms are essentially identical with known axiomatizations in many examples.

\medskip \noindent In Sections~\ref{section-tame-topology}, \ref{section:aleph-base}, and \ref{sec:stable-categorical}, we focus on more specific settings applicable to the examples listed above. We show that under further hypotheses on the theories and the notion of dimension, the general conditions of Theorem~\ref{thm:mainsec3} specialize to more familiar notions. 

\medskip \noindent Section~\ref{section-tame-topology} treats several settings where $T_\cap$ is equipped with natural topology compatible with the aforementioned dimension; the use of the term ``pseudo-topological'' is motivated by consideration of these special cases.
When $T_\cap$ is o-minimal and $\dim$ is the canonical o-minimal dimension, for example, any theory extending $T_\cap$ defines pseudo-denseness, and $T_i$ has pseudo-closures in $T_\cap$ if and only if $T_\cap$ is an \textit{open core} of $T_i$, i.e., the closure of any $\sM_i$-definable set is already $\sM_\cap$-definable.
This gives the following result. 

\begin{thm} 
\label{thm: mainsection5}
Suppose $T_\cap$ is o-minimal.
If $T_\cap$ is an open core of each $T_i$ then $T^{*}_\cup$ exists.
\end{thm}

\noindent 
In Section~\ref{section:aleph-base}, we show that if $T_\cap$ is $\aleph_0$-stable, and $\dim$ is Morley rank, then any theory extending $T_\cap$ has pseudo-closures in $T_\cap$.
The \textbf{induced dimension} of a definable set $X$ in a model of $T_i$ is the Morley rank of any pseudo-closure of $X$.
Assuming further that $T_\cap$ defines multiplicity, we show that $T_i$ defines pseudo-denseness if and only if $T_i$ uniformly defines induced dimension. 
Theorem~\ref{thm:3} applies to the example of algebraically closed fields with independent valuations.

\begin{thm}\label{thm:3}
Suppose $T_\cap$ is $\aleph_0$-stable and defines multiplicity.
If each $T_i$ defines induced dimension, then $T^{*}_\cup$ exists.
\end{thm}

\noindent In Section~\ref{sec:stable-categorical} we consider the case when $T_\cap$ is $\aleph_0$-stable, $\aleph_0$-categorical, and weakly eliminates imaginaries. We prove that $T_i$ defines pseudo-denseness if and only if $T_i$ eliminates $\exists^\infty$.
This applies to the examples of generic predicates, generic Skolemizations, and the random graph, hypergraph, and tournament.
It also generalizes Winkler's result on model companions of disjoint unions of theories~\cite{Winkler}.

\begin{thm}\label{thm:2}
Suppose $T_\cap$ is complete, $\aleph_0$-stable, and $\aleph_0$-categorical.
If $T_i^{\eq}$ eliminates $\exists^\infty$ for all $i$, then $T^*_\cup$ exists.
If $T_\cap$ weakly eliminates imaginaries and each $T_i$ eliminates $\exists^\infty$, then $T^{*}_\cup$ exists.
\end{thm}

\noindent In \cite[3.1.20]{Lou-thesis} van den Dries notes a similarity between his main result and Winkler's theorem and claims that this similarity \textit{``\ldots suggests a common generalization of Winkler's and my results''.}
We believe the present paper provides a moral answer to this suggestion but perhaps not the final answer, as
 our results do not in fact generalize the main result of \cite{Lou-thesis}.

\subsection{Conventions and notation} 
Throughout, $m$ and $n$ range over the natural numbers (containing $0$), and $k$ and $l$ range over the integers.
We work in multi-sorted first-order logic.
Our semantics allows empty sorts and empty structures. Our syntax includes logical constants $\top$ and $\bot$ interpreted as true and false, respectively. 
We view constant symbols as $0$-ary function symbols.

\medskip\noindent  Throughout, $L$ is a language with $S$ the set of sorts. Concepts like variables, functions, formulas, etc.\ are by default with respect to $L$.
Suppose $\sM$ is an 
$L$-structure. We use the corresponding capital letter $M$ to denote the $S$-indexed family $(M_s)_{s \in S}$ of underlying sets of the sorts of $\sM$.
By $A \subseteq M$, we mean $A =(A_s)_{s \in S}$ with $A_s \subseteq M_s$ for each $s \in S$.
If $A \subseteq M$, then a tuple of elements (possibly infinite) in $A$ is a tuple whose each component is in $A_s$ for some $s \in S$. If $x = (x_j)_{j\in J}$ is a tuple of variables (possibly infinite), we let $A^x = \prod_{j\in J} A_{s(x_j)} $ where $s(x_j)$ is the sort of the variable $x_j$.
If $\varphi(x,y)$ is an $L$-formula and $b \in M^y$, we let $\varphi(\sM, b)$ be the set defined in $\sM$ by the $L(b)$-formula $\varphi(x,b)$.
We call such $\varphi(\sM, b)$ a definable set in $\sM$ or an $\sM$-definable set. Hence, ``definable'' means ``definable, possibly with parameters''.
If we wish to exclude parameters, we write ``$\emptyset$-definable''. 

\medskip\noindent Whenever we consider multiple reducts of a structure, we decorate these reducts with the same decorations as their languages. For example, if $L_0\subseteq L_1$ are languages, we denote an $L_1$-structure by $\sM_1$, and we denote its reduct $\sM_1|_{L_0}$ to $L_0$ by $\sM_0$.
In this situation, we write ``in $\sM_0$'' to denote that we are evaluating some concept in the reduct.

\subsection*{Acknowledgement} 
We would like to thank Anand Pillay and Pierre Simon for pointing to us useful known results.
The referee's comments were particularly helpful in shaping the current form of the paper.

\section{Interpolative fusions and model companions} 

\subsection{Basic results}\label{sec:fusions}

\noindent This section clarifies the relationship between interpolative fusions and model companions of unions of theories. We make use of the definitions and notation set in the introduction.

\medskip \noindent The name ``interpolative fusion'' is inspired by a connection to the classical Craig interpolation theorem, which we state below; a proof is given, for example, in \cite[Theorem 6.6.3]{Hodges}.
It is well-known that in the context of first-order logic, the Craig interpolation theorem is equivalent to Robinson's joint consistency theorem. 

\begin{thm}\label{thm:CLR0}
Suppose $L_1$ and $L_2$ are first order languages with intersection $L_\cap$ and $\varphi_i$ is an $L_i$-sentence for $i \in \{1,2\}$.
If $\models (\varphi_1 \rightarrow \varphi_2)$ then there is an $L_\cap$-sentence $\psi$ such that $\models (\varphi_1 \rightarrow \psi)$ and $\models (\psi \rightarrow \varphi_2)$.
Equivalently: $\{ \varphi_1, \varphi_2 \}$ is inconsistent if and only if there is an $L_\cap$-sentence $\psi$ such that $\models (\varphi_1 \rightarrow \psi)$ and $\models (\varphi_2 \rightarrow \neg \psi)$.
\end{thm}

\noindent Our first result is an easy generalization of Theorem~\ref{thm:CLR0}, applicable to our setting.

\begin{cor} \label{cor:interpolation}
For each $i\in I$, let $\Sigma_i(x)$ be a set of $L_i$-formulas. If $\bigcup_{i\in I} \Sigma_i(x)$ is inconsistent, then there is a finite subset $J\subseteq I$ and an $L_\cap$-formula $\varphi^i(x)$ for each $i\in J$ such that:
\[  \Sigma_i(x) \models \varphi^i(x) \text{ for all } i \in J, \text{and } \{ \varphi^i(x)\mid i\in J \} \text{  is inconsistent}. \]
\end{cor}

\begin{proof}
By introducing a new constant symbol for each free variable, we reduce to the case when $x$ is the empty tuple of variables.
We may also assume that the sets $\Sigma_i$ are closed under conjunction. 
By compactness, if $\bigcup_{i\in I} \Sigma_i$ is inconsistent, then there is a nonempty finite subset $J\subseteq I$ and a formula $\varphi_i\in \Sigma_i$ for all $i\in J$ such that $\{\varphi_i\mid i\in J\}$ is inconsistent.

We argue by induction on the size of $J$. For the sake of notational simplicity, we suppose $J = \{1,\ldots,n\}$.
If $n = 1$, then $\{\varphi_1\}$ is inconsistent, and we choose $\varphi^1$ to be the contradictory $L_\cap$-formula $\bot$. Suppose $n \geq 2$.
Then $(\varphi_1 \land \ldots \land \varphi_{n-1})$ is an $(L_1 \cup \ldots \cup L_{n-1})$-sentence and the set
\[\{ (\varphi_1 \land \ldots \land \varphi_{n-1}), \varphi_n\} \text{ is inconsistent}. \]
Applying Theorem \ref{thm:CLR0}, we get a sentence $\psi$ in $L_n \cap (L_1 \cup \ldots \cup L_{n-1}) = L_\cap$  such that
\[ \models (\varphi_1 \land \ldots \land \varphi_{n-1}) \rightarrow \psi \quad \text{and}  \quad \models \varphi_{n} \rightarrow \neg \psi.  \]
Then $\varphi_i\wedge \neg \psi$ is an $L_i$-sentence for all $1 \leq i \leq n-1$, and $\{\varphi_i  \wedge \neg \psi\mid 1\leq i\leq n-1\}$ is inconsistent.
Applying induction, we choose for each $1 \leq i \leq n-1$ an $L_\cap$-sentence $\theta^i$ such that 
\[ \models (\varphi_i  \wedge \neg \psi)  \rightarrow \theta^i \text{ for all } 1 \leq i \leq n-1, \text{and } \models \neg (\theta^1 \land \ldots \land \theta^{n-1}).  \]
Finally, set $\varphi^i$ to be $(\psi \vee \theta^i)$ for $1 \leq i \leq n-1$, and set $\varphi^{n} $ to be $\neg \psi$. 
It is easy to check that all the desired conditions are satisfied.
\end{proof}

\noindent Corollary~\ref{cor:consistency} follows immediately from Corollary~\ref{cor:interpolation} and generalizes Robinson's joint consistency theorem.

\begin{cor}\label{cor:consistency}
Let $p(x)$ 
be a complete $L_\cap$-type, and for all $i\in I$, let $p_i(x)$ be a complete $L_i$-type such that $p(x)\subseteq p_i(x)$. Then $\bigcup_{i\in I} p_i(x)$ is consistent.
\end{cor}

\noindent Corollary~\ref{cor:interpolation} also allows us to show that families of definable sets that are not separated have ``potentially'' non-empty intersection.

\begin{lem}\label{lem:sem-craig}
Let $\sM_\cup$ be an $L_\cup$-structure, and suppose $J\subseteq I$ is finite and $X_i\subseteq M^x$ is $\sM_i$-definable for all $i\in J$.  The family $(X_i)_{i \in J}$ is separated if and only if for every $L_\cup$-structure $\sN_\cup$ such that $\sM_i \elesub \sN_i$ for all $i \in I$, $\bigcap_{i \in J} X_i(\sN_\cup) = \emptyset$.
\end{lem}

\begin{proof}
Suppose $(X_i)_{i \in J}$ is separated.
Then there are $\sM_\cap$-definable $X^1,\ldots,X^n$ such that $X_i \subseteq X^i$ for all $i \in J$ and $\bigcap_{i \in J} X^n = \emptyset$.
Suppose $\sN_\cup$ is a $T_\cup$-model satisfying $\sM_i \elesub \sN_i$ for all $i \in I$.
Then $X_i(\sN_\cup) \subseteq X^i(\sN_\cup)$ for all $i \in J$ and $\bigcap_{i \in J} X^i(\sN_\cup) = \emptyset$, so also $\bigcap_{i\in J} X_i(\sN_\cup) = \emptyset$.

Conversely, suppose that $\bigcap_{i \in J} X_i(\sN_\cup) = \emptyset$ for every $L_\cup$-structure $\sN_\cup$ such that $\sM_i \elesub \sN_i$ for all $i \in I$. For each $i\in J$, let $\varphi_i(x,b)$ be an $L_i(M)$-formula defining $X_i$. Then the partial type \[\bigcup_{i\in I} \ediag(\sM_i)\cup \bigcup_{i\in J} \varphi_i(x,b) \quad  \text{ is inconsistent.}\]
By compactness, there is a finite subset $J'\subseteq I$ with $J \subseteq J'$, a finite tuple $c\in M^y$ and a formula $\psi_i(b,c)\in \ediag(\sM_i)$ for each $i\in J'$ such that \[\{\psi_i(b,c)\mid i\in J'\}\cup \{\varphi_i(x,b)\mid i\in J\}\quad  \text{ is inconsistent. }\]
Let $\varphi_i$ be the true formula $\top$ when $i\in J'\setminus J$, and define $\varphi'_i(x,y,z) = \varphi_i(x,y)\land \psi_i(y,z)$ for all $i\in J'$. Note that since $\sM_i\models \psi_i(b,c)$, 
\[ \varphi_i(\sM_\cup,b)  =  \varphi'_i(\sM_\cup,b,c).  \]
Applying Corollary~\ref{cor:interpolation}, we obtain an inconsistent family $\{\theta_i(x,y,z)\mid i\in J'\}$ of $L_\cap$-formulas such that $\models \varphi'_i(x,y,z)\rightarrow \theta_i(x,y,z)$ for each $i\in J'$.
It follows that $$\varphi_i(\sM_\cup,b,c)\subseteq \theta_i(\sM_\cup,b,c) \text{ for all }i\in J',  \text{ and } \bigcap_{i\in J'} \theta_i(\sM_\cup,b,c) = \emptyset.$$
But since $\varphi_i(\sM_\cup,b,c) = M^x$ when $i\in J'\setminus J$, also $\theta_i(\sM_\cup,b,c) = M^x$ when $i\in J'\setminus J$. So $\bigcap_{i\in J}\theta_i(\sM_\cup,b,c) = \emptyset$ and $(\theta_i(\sM_\cup,b,c))_{i\in J}$ separates $(X_i)_{i\in J}$. 
\end{proof}

\begin{rem} \label{rem: robust}
If we change languages in a way that does not change the class of definable sets (with parameters), then the class of interpolative $L_\cup$-structures is not affected. 
In particular:
\begin{enumerate}
    \item An interpolative structure $\sM_\cup$ remains so after adding new constant symbols naming elements of $M$ to each of the languages $L_\square$ for $\square\in I\cup \{\cup,\cap\}$. 
    \item Suppose $L^{\diamondsuit}_\square$ is an expansion by definitions of $L_\square$ for $\square \in I\cup \{ \cap\}$, $L^{\diamondsuit}_i \cap L^{\diamondsuit}_j = L^{\diamondsuit}_\cap$ for distinct $i$ and $j$ in $I$, and $L^{\diamondsuit}_\cup = \bigcup_{i\in I} L^\diamondsuit_i$ is the resulting expansion by definitions of $L_\cup$. Then any $L_\cup$-structure $\sM_\cup$ has a canonical expansion $\sM_\cup^\diamondsuit$ to an $L_\cup^\diamondsuit$-structure. And  $\sM_\cup$ is an interpolative $L_\cup$-structure if and only if $\sM_\cup^{\diamondsuit}$ is an interpolative $L^{\diamondsuit}_\cup$-structure.
    \item An interpolative $\sM_\cup$-structure remains so after replacing each function symbol $f$ in each of the languages $L_\square$ for $\square\in I\cup \{\cup,\cap\}$ by a relation symbol $R_f$, interpreted as the graph of the interpretation of $f$ in $\sM_\cup$.
    \item Suppose $\sM_\cup$ is an $L_\cup$-structure.
    Moving to $\sM_\cap^\eq$ involves the introduction of new sorts and function symbols for quotients by $L_\cap$-definable equivalence relations on $M$.
    For all $\square\in I\cup \{\cup,\cap\}$, let $L_\square^{\cap-\eq}$ be the language expanding $L_\square$ produced by adding new symbols for $L_\square$-definable equivalence relations, and let $\sM_\square^{\cap-\eq}$ be the natural expansion of $\sM_\square$ to $L_\square^{\cap-\eq}$.
    Then $\sM_\cup$ is interpolative if and only if $\sM_\cup^{\cap-\eq}$ is interpolative. This follows from the fact that if $X_\square$ is an $\sM_\square^{\cap-\eq}$-definable set in one of the new sorts, corresponding to the quotient of $M^x$ by an $L_\cap$-definable equivalence relation, then the preimage of $X_\square$ under the quotient is $\sM_\square$-definable.
\end{enumerate} 
\end{rem}

\noindent We now show that interpolative models of $T_\cup$ can be thought of as  ``relatively existentially closed'' models of  $T_\cup$, and the interpolative fusion $T_\cup^*$ can be thought of as the ``relative model companion'' of $T_\cup$.

\medskip \noindent 
Recall that a theory $T$ is {\bf inductive} if the class of models of $T$ is closed under directed unions.
Equivalently, $T$ admits an axiomatization by $\forall\exists$-sentences. 

\begin{fact}[\cite{Hodges} Theorem 8.3.6]\label{fact:inductive}
Suppose $T$ is inductive.
Then $T$ has a model companion if and only if the class of existentially closed $T$-models is elementary. If these equivalent conditions hold, then the model companion of $T$ is the theory of existentially closed $T$-models.
\end{fact}

\begin{thm} \label{thm:relativemc}
Suppose $\sM_\cup\models T_\cup$. 
\begin{enumerate}
    \item $\sM_\cup$ is an interpolative structure if and only if for all $\sN_\cup$ such that $\sM_i\elesub \sN_i$ for all $i\in I$, 
$$ \sN_\cup \models \exists x\, \varphi_\cup(x) \quad \text{implies}  \quad \sM_\cup \models \exists x\, \varphi_\cup(x)$$ whenever $\varphi_\cup(x)$ is a Boolean combination of $L_i$-formulas with parameters from $M$.
\item If each $L_i$ is relational and each $T_i$ is model-complete, then the interpolative models of $T_\cup$ are exactly the existentially closed models, and the interpolative fusion of $T_\cup$ is precisely the model companion of $T_\cup$, if either of these exists.
\item There exists an interpolative structure $\sN_\cup$ such that $\sM_\cup\subseteq \sN_\cup$, and $\sM_i\elesub \sN_i$ for all $i\in I$. 
\item If $T^*_\cup$ exists, $\sM_\cup \models T^*_\cup$, $\sN_\cup \models T^*_\cup$, and $\sM_i\elesub \sN_i$ for all $i\in I$, then $\sM \elesub \sN$.
\end{enumerate}
\end{thm}
\begin{proof}
Part (1) follows immediately from  Lemma~\ref{lem:sem-craig} and the definition of interpolative structure. 

For part (2), since each $T_i$ is model-complete, whenever $\sM_\cup\subseteq \sN_\cup$ are both models of $T_\cup$, we have $$\sM_i\elesub \sN_i \quad \text{for all } i\in I.$$ 
As $L_\cup$ is relational, no atomic formula contains symbols from distinct languages, and hence every quantifier-free $L_\cup$-formula is a Boolean combination of $L_i$-formulas. Therefore, it follows from (1) that $\sM_\cup$ is interpolative if and only if it is existentially closed in the class of $T_\cup$-models. 
Each $T_i$ is model-complete and hence inductive, so $T_\cup$ is also inductive. By Fact~\ref{fact:inductive} that $T^*_\cup$ is the model companion of $T_\cup$, if either either of these exists.

For parts (3) and (4), applying Remark~\ref{rem: robust}, we can assume by Morleyizing that each $T_i$ admits quantifier elimination and each $L_i$ is relational.

Now (3) follows from the well-known fact that every model of an inductive theory embeds in an existentially closed model~\cite[Theorem 8.2.1]{Hodges}. In particular, for all $\sM_\cup\models T_\cup$, there exists $\sN_\cup$ such that $\sM_\cup\subseteq \sN_\cup$ and $\sN_\cup\models T_\cup$ is existentially closed. Then $\sN_\cup$ is interpolative by (2), and $\sM_i\preceq \sN_i$ for all $i$ by quantifier elimination.

For (4), if $T_\cup^*$ exists, then $T_\cup^*$ is model-complete by (2). Since $\sM_i$ is an $L_i$-substructure of $\sN_i$ for all $i\in I$, $\sM_\cup$ is an $L_\cup$-substructure of $\sN_\cup$, so $\sM_\cup\preceq \sN_\cup$. 
\end{proof}

\noindent In the rest of this section, we will do a bit more work to improve Theorem~\ref{thm:relativemc}(2) by removing the hypothesis that the languages $L_i$ are relational.

\medskip \noindent  A formula is \textbf{atomic flat} if it is of the form $x_1 = x_2$,  $R(x_1,\ldots,x_n)$, or $f(x_1,\ldots,x_n) = x_{n+1}$, where $R$ is an $n$-ary relation symbol and $f$ is an $n$-ary function symbol. Here  $x_1,\dots,x_{n+1}$ are arbitrary variables, which need not be distinct.  A \textbf{flat literal} is an atomic flat formula or the negation of an atomic flat formula.  
A \textbf{flat formula} is a conjunction of finitely many flat literals.
An \textbf{E$\flat$-formula} is a formula of the form $\exists y\, \varphi(x,y)$, where $\varphi(x,y)$ is flat and $\models \forall x\, \exists^{\leq 1} y\, \varphi(x,y)$. Here $x$ and $y$ may be tuples of variables.

\begin{rem}\label{rem:flat-closure}
The class of E$\flat$-formulas is closed (up to equivalence) under finite conjunction: the conjunction of the E$\flat$-formulas $\exists y_1\, \varphi_1(x,y_2)$ and $\exists y_2\, \varphi_2(x,y_2)$ is equivalent to the E${\flat}$-formula \[\exists y_1y_2\, (\varphi_1(x,y_1)\land \varphi_2(x,y_2)).\]
\end{rem}

\noindent Lemma~\ref{lem:eflat} is essentially \cite[Thm 2.6.1]{Hodges}. 
Hodges uses ``unnested'' for ``flat''.

\begin{lem} \label{lem:eflat}
Every literal (atomic or negated atomic formula) is logically equivalent to an E$\flat$-formula.
\end{lem}

\begin{proof}
We first show that for any term $t(x)$, with variables $x = (x_1,\ldots,x_n)$, there is an associated E$\flat$-formula $\varphi_t(x,y)$ such that $\varphi_t(x,y)$ is logically equivalent to $ t(x) = y$.
We apply induction on terms.
For the base case where $t(x)$ is the variable $x_k$, we let $\varphi_t(x,y)$ be $x_k = y$. Now suppose $t_1(x),\ldots,t_m(x)$ are terms and $f$ is an $m$-ary function symbol.
Then $\varphi_{f(t_1,\ldots,t_m)}$ is the E$\flat$-formula equivalent to
\[ \exists z_1\ldots z_m\, \left[\bigwedge_{i = 1}^{m} \varphi_{t_i}(x,z_i) \land (f(z_1,\ldots,z_m) = y) \right]. \]

We now show that every atomic or negated atomic formula is equivalent to an E$\flat$-formula.
Suppose $t_1(x),\ldots,t_m(x)$ are terms  and $R$ is either an $m$-ary relation symbol or $=$ (in the latter case, we have $m = 2$).
Then the atomic formula
$R(t_1(x),\ldots,t_m(x))$ is equivalent  to 
\[
\exists y_1 \ldots  y_m\, \left[ \bigwedge_{i = 1}^{m} \varphi_{t_i}(x,y_i) \land R(y_1,\ldots,y_m) \right].
\]
Negated atomic formulas can be treated similarly.  
\end{proof}

\begin{cor}\label{cor:eflat}
Every quantifier-free formula is logically equivalent to a finite disjunction of E$\flat$-formulas.
\end{cor}
\begin{proof}
Suppose $\varphi(x)$ is quantifier-free. Then $\varphi(x)$ is equivalent to a formula in disjunctive normal form, i.e., a finite disjunction of finite conjunctions of literals. Applying Lemma~\ref{lem:eflat} to each literal and using Remark~\ref{rem:flat-closure}, we find that $\varphi(x)$ is equivalent to a finite disjunction of E$\flat$-formulas. 
\end{proof}

\begin{rem}\label{rem:Lcupflat}
Any flat literal $L_\cup$-formula is an $L_i$-formula for some $i\in I$. As a consequence, if $\varphi(x)$ is a flat $L_\cup$-formula, then there is some finite $J\subseteq I$ and a flat $L_i$-formula $\varphi_i(x)$ for all $i\in J$ such that $\varphi(x)$ is logically equivalent to $\bigwedge_{i\in J}\varphi_i(x)$.
\end{rem}

\noindent We obtain a restatement of Theorem~\ref{thm: first Theorem} from the introduction:

\begin{thm} \label{Thm: IFvsEC}
Suppose each $T_i$ is model-complete.
Then $\sM_\cup \models T_\cup$ is interpolative if and only if it is existentially closed in the class of $T_\cup$-models. Hence, $T^*_\cup$ is precisely the model companion of $T_\cup$, if either of these exists.
\end{thm}

\begin{proof}
We prove the first statement. Let $\sM_\cup \models T_\cup$ be existentially closed.
Suppose $J \subseteq I$ is finite and $\varphi_i(x)$ is an $L_i(M)$-formula for each $i \in J$ such that $( \varphi_i(\sM_\cup))_{i \in J}$ is not separated.
We may assume each $\varphi_i(x)$ is existential, as $T_i$ is model-complete.
Lemma~\ref{lem:sem-craig} gives a $T_\cup$-model $\sN_\cup$ extending $\sM_\cup$ such that $\sN_\cup \models \exists x\, \bigwedge_{i \in J} \varphi_i(x)$.
As $\sM_\cup$ is existentially closed and each $\varphi_i$ is existential, we have $\sM_\cup \models \exists x\, \bigwedge_{i \in J} \varphi_i(x)$.
Thus $\sM_\cup$ is interpolative.

Now suppose $\sM_\cup \models T_\cup$ is interpolative.
Suppose $\psi(x)$ is a quantifier-free $L_\cup(M)$-formula and $\sN_\cup$ is a $T_\cup$-model extending $\sM_\cup$ such that $\sN_\cup \models \exists x\, \psi(x)$.
Applying Corollary~\ref{cor:eflat}, $\psi(x)$ is logically equivalent to a finite disjunction of E$\flat$-formulas $\bigvee_{k=1}^n \exists y_k\, \psi_k(x,y_k)$. Then for some $k$, $\sN_\cup \models \exists x\, \exists y_k\, \psi_k(x,y_k)$. By Remark~\ref{rem:Lcupflat}, the flat $L_\cup(M)$-formula $\psi_k(x,y_k)$ is equivalent to a conjunction $\bigwedge_{i \in J} \varphi_i(x,y_k)$ where $J \subseteq I$ is finite and $\varphi_i(x,y_k)$ is a flat $L_i(M)$-formula for each $i \in J$.
So $\sN_\cup \models \exists x\,\exists y_k\, \bigwedge_{i \in J} \varphi_i(x,y_k)$.
As each $T_i$ is model-complete, we have $\sM_i \elesub \sN_i$ for all $i \in I$.
By Lemma~\ref{lem:sem-craig}, the sets defined by $\varphi_i(x,y_k)$ are not separated, and since $\sM_\cup$ is interpolative, $\sM_\cup \models \exists x\,\exists y_k\, \bigwedge_{i \in J} \varphi_i(x,y_k)$.
So $\sM_\cup  \models \exists x\, \psi(x)$.  This shows $\sM_\cup$ is existentially closed. 

Each $T_i$ is model-complete and hence inductive, so $T_\cup$ is inductive. Using Fact~\ref{fact:inductive}, we get the second statement as a consequence of the first statement.
\end{proof}

\subsection{Existential interpretations}\label{sec:interp}

\noindent In this section, we recall some standard facts about interpretations. We pay special attention to the case of existential bi-interpretations between inductive theories, which preserve the existence of model companions. This allows us to sometimes identify the model companion $T^*$ of a theory $T$ of interest with an interpolative fusion of simpler theories, via a bi-interpretation. We keep the notational conventions from the introduction.

\medskip \noindent Let $T$ be an $L$-theory and $T'$ be an $L'$-theory. An \textbf{interpretation} of $T'$ in $T$, $F\colon T\rightsquigarrow T'$, consists of the following data: 
\begin{enumerate}
    \item For every sort $s'$ in $L'$, an $L$-formula $\varphi_{s'}(x_{s'})$ and an $L$-formula $E_{s'}(x_{s'},x^*_{s'})$. 
    \item For every relation symbol $R'$ in $L'$ of type $(s'_1,\dots,s'_n)$ in $L'$, an $L$-formula $\varphi_{R'}(x_{s'_1},\dots,x_{s'_n})$.
    \item For every function symbol $f'$ in $L'$ of type $(s'_1,\dots,s'_n)\to s'$ in $L'$, an $L$-formula $\varphi_{f'}(x_{s'_1},\dots,x_{s'_n},x_{s'})$. 
\end{enumerate}
We then require that for every model $\sM\models T$, the formulas above define an $L'$-structure $\sM'\models T'$ in the natural way. See~\cite[Section 5.3]{Hodges} for details.  For every sort $s'$ in $L'$, the underlying set $M'_{s'}$ of the $s'$ sort in $\sM'$ is the quotient of  $\varphi_{s'}(\sM)$ by the equivalence relation defined by $E_{s'}$. We write $\pi_{s'}$ for the surjective quotient map $\varphi_{s'}(\sM)\to M'_{s'}$.
We sometimes denote $\sM'$ by $F(\sM)$.

\medskip \noindent An interpretation $F\colon T\rightsquigarrow T'$ is an \textbf{existential  interpretation} if for each sort $s'$ in $L'$, the $L$-formula $\varphi_{s'}(x_{s'})$ is $T$-equivalent to an existential formula, and all other formulas involved in the interpretation and their negations (i.e., the formulas $E_{s'}$, $\neg E_{s'}$, $\varphi_{R'}$, $\neg\varphi_{R'}$, $\varphi_{f'}$,  and $\neg\varphi_{f'}$) are also $T$-equivalent to existential formulas.

\begin{lem}\label{lem:pullback} Suppose $F\colon T\rightsquigarrow T'$ is an existential interpretation. Let $\varphi'(y)$ be a quantifier-free $L'$-formula, where $y = (y_1,\dots,y_n)$ and $y_i$ is a variable of sort $s'_i$. Then there is an existential  $L$-formula $\widehat{\varphi}(x_{s'_1},\dots,x_{s'_n})$ such that for every $\sM\models T$ and every tuple $a = (a_1,\dots,a_n)$ with $a_i\in \varphi_{s'_i}(\sM)$, $\sM\models \widehat{\varphi}(a)$ if and only if $F(\sM)\models \varphi'(\pi_{s_1}(a_1),\dots,\pi_{s_n}(a_n))$. 
\end{lem}
\begin{proof}
By Corollary~\ref{cor:eflat}, $\varphi'(y)$ is equivalent to a finite disjunction of E$\flat$-formulas.
In the proof of ~\cite[Theorem 5.3.2]{Hodges}, Hodges gives an explicit translation from $L'$-formulas to $L$-formulas.
We apply this translation and observe that because $F$ is an existential interpretation, the translation of a finite disjunction of E$\flat$-formulas is an existential $L$-formula.
\end{proof}

\noindent A {\bf bi-interpretation} $(F,G, \eta, \eta')$ between $T$ and $T'$ consists of an interpretation $F\colon T\rightsquigarrow T'$ and an interpretation  $G\colon T'\rightsquigarrow T$, together with $L$-formulas and $L'$-formulas (one for each sort in $L$ and $L'$, respectively) defining for each $\sM \models T$ and each $\sN' \models T'$ isomorphisms $$\eta_{\sM}: \sM \to G(F(\sM)) \quad \text{and}  \quad \eta'_{\sN'}: \sN' \to F(G(\sN')).$$ See~\cite[Section 5.4]{Hodges} for the precise definition.
Such a bi-interpretation is  \textbf{existential} if $F$ and $G$ are each existential interpretations, and moreover the $L$-formulas and $L'$-formulas defining the isomorphisms are existential. If there is an existential bi-interpretation between $T$ and $T'$, we say that $T$ and $T'$ are existentially bi-interpretable. 
The next lemma is \cite[Exercise 5.4.3]{Hodges}.

\begin{lem} \label{lem: existentialbi-interpretation}
Suppose  $F\colon T\rightsquigarrow T'$ is existential. Then $F$ induces a functor from the category of models of $T$ and embeddings to the category of models of $T'$ and embeddings.
Suppose moreover that $(F,G, \eta, \eta')$ is an existential bi-interpretation between $T$ and $T'$. Then the induced functors form an equivalence of categories; in particular, for every $L$-embedding $f: \sM \to \sN$, the following diagram, which expresses that $\eta$ is a natural isomorphism from the identity functor to $G\circ F$, commutes:
\[\xymatrix{
\sN \ar[rr]^{\eta_{\sN}} & & G(F(\sN))\\
\sM \ar[u]^f\ar[rr]^{\eta_{\sM}} & & G(F(\sM))\ar[u]_{G(F(f))}
}\]
\end{lem}

\begin{prop}\label{prop:delta1}
Suppose $(F, G, \eta, \eta')$ is an existential bi-interpretation between $T$ and $T'$. Then $\sM$ is an existentially closed model of $T$ if and only if $F(\sM)$ is an existentially closed model of $T'$.
\end{prop}
\begin{proof}
It suffices to show that if $F(\sM)$ is an existentially closed model of $T'$, then $\sM$ is an existentially closed model of $T$. Indeed, by symmetry it follows that if $G(\sN')$ is an existentially closed model of $T$, then $\sN'$ is an existentially closed model of $T'$. And then, since $\eta_{\sM}:\sM \to G(F(\sM))$ is an isomorphism, if $\sM$ is existentially closed, then $F(\sM)$ is existentially closed.

So assume that $F(\sM)$ is an existentially closed model of $T'$. Let $f\colon \sM\to \sN$ be an embedding of $T$-models, and let $\varphi(y)$ be a quantifier-free formula with parameters from $\sM$ that is satisfied in $\sN$. By commutativity of the diagram in Lemma \ref{lem: existentialbi-interpretation}, after moving the parameters of $\varphi(y)$ into $G(F(\sM))$ by the isomorphism $\eta_\sM$, we find that $\varphi(y)$ is satisfied in $G(F(\sN))$, and it suffices to show that it is satisfied in $G(F(\sM))$.

By Lemma~\ref{lem:pullback}, there is an existential $L'$-formula $\widehat{\varphi}'(x)$ with parameters from $F(\sM)$ such that $F(\sN)\models \widehat{\varphi}'(a)$ if and only if $G(F(\sN))\models \varphi(b)$, where $b$ is the image of $a$ under the appropriate $\pi_s$ quotient maps. Writing $\widehat{\varphi}'(x)$ as $\exists z\, \psi'(x,z)$, we have $F(\sN)\models \psi'(a,c)$ for some $c$, where $a$ is any preimage of the tuple from $G(F(\sN))$ satisfying $\varphi(y)$. But since $F(\sM)$ is existentially closed, there are some $a^*$ and $c^*$ in $F(\sM)$ such that $F(\sM)\models \psi'(a^*,c^*)$, so $F(\sM)\models \widehat{\varphi}'(a^*)$, and it follows that $\varphi(y)$ is satisfied in $G(F(\sM))$, as desired. 
\end{proof}

\begin{cor}\label{cor:delta1-cor}
Suppose $T$ and $T'$ are inductive, and $(F,G, \eta, \eta')$ is an existential bi-interpretation between $T$ and $T'$. Then $T$ has a model companion $T^*$ if and only if $T'$ has a model companion $(T')^*$. Further, $(F,G, \eta, \eta')$ induces an existential bi-interpretation between $T^*$ and $(T')^*$ when they exist.

In particular, if $T_i$ is model-complete for all $i\in I$, $T$ is an inductive theory, and $(F,G, \eta, \eta')$ is an existential bi-interpretation between $T$ and $T_\cup$, then $T$ has a model companion $T^*$ if and only if the interpolative fusion $T_\cup^*$ exists. Further, $(F,G, \eta, \eta')$ induces an existential bi-interpretation between $T^*$ and $T_\cup^*$ when they exist.
\end{cor}

\begin{proof}
Suppose $T$ has a model companion $T^*$. By~\cite[Theorem 5.3.2]{Hodges}, for every $L$-sentence $\varphi\in T^*$, there is an $L'$-sentence $\varphi'$, such that for all $\sM'\models T'$, $\sM'\models \varphi'$ if and only if $G(\sM)\models \varphi$. Let $(T')^* = T'\cup \{\varphi'\mid \varphi\in T^*\}$. Then $\sM'\models (T')^*$ if and only if $\sM'\models T'$ and $G(\sM')\models T^*$. By Proposition~\ref{prop:delta1} and Fact~\ref{fact:inductive}, $\sM'\models (T')^*$ if and only if $\sM'$ is an existentially closed model of $T'$. So $(T')^*$ is the model companion of $T'$ by Fact~\ref{fact:inductive}. Proposition~\ref{prop:delta1} further implies that $\sM\models T^*$ if and only if $F(\sM)\models (T')^*$. So $(F,G, \eta, \eta')$ induces an existential bi-interpretation between the model companions. 

The application to interpolative fusions then follows immediately from the first statement and Theorem~\ref{Thm: IFvsEC}.
\end{proof}

\subsection{Examples}\label{sec:examples}
 
We now give a more detailed description of the items of Examples~\ref{ex:exampleblock1} and \ref{ex:exampleblock2}.
The first four examples are interpolative fusions from the literature.
The last four are well-known theories that we show are bi-interpretable with interpolative fusions of simpler theories. 
All of these examples are natural, as each is either the model companion of a natural theory or is the theory of an interpolative structure found in the wild.

\medskip\noindent \textbf{Generic predicates and disjoint unions of theories.} The simplest case of the notion of interpolative fusion is when the $L_i$ are pairwise disjoint. This case was treated by Winkler~\cite{Winkler}.
In our notation, Winkler showed that if $L_\cap$ is the empty language, $T_\cap$ is the theory of an infinite set, and each $T_i$ is model-complete and eliminates $\exists^\infty$, then $T_\cup$ has a model companion, which is $T_\cup^*$ by Theorem~\ref{thm: first Theorem}.
This result easily generalizes via Morleyization to show that $T^*_\cup$ exists when each $T_i$ eliminates $\exists^\infty$ and $T_\cap$ is the theory of an infinite set.

\medskip\noindent If $I = \{1,2\}$ and $T_2$ is the theory of an infinite set equipped with an infinite and co-infinite unary predicate, then $T^*_\cup$ is well-known as the expansion of $T_1$ by a generic unary predicate (see~\cite{Cha-Pi}).

\medskip\noindent  \textbf{Algebraically closed fields with independent valuations.} Let $T_\cap$ be the theory of algebraically closed fields and $T_i$ be the theory of an algebraically closed field equipped with a non-trivial valuation $v_i$  for each $i \in I$. (We let $L_i$ be the language of fields extended by a unary predicate for the valuation ring of  $v_i$ for each $i \in I$.)
Then  $T_\cup$ has a model companion $T^*_\cup$, which is the theory of an algebraically closed field $K$ equipped with a family $(v_i)_{i \in I}$ of pairwise independent valuations.
The theory $T^*_\cup$ is studied in \cite{Lou-thesis} and \cite{Johnson-thesis}.

\medskip\noindent \textbf{The group of integers with $p$-adic valuations.} Given integers $k,l$ we write $k \elesub_p l$ if the $p$-adic valuation of $k$ is no greater then the $p$-adic valuation of $l$.
Let $I$ be the set of primes, $T_\cap$ be the theory of $(\zz,+)$, and $T_p$ be the theory of $(\zz,+,\elesub_p)$ for $p \in I$.
Then the theory of $(\zz,+,(\elesub_p)_{p \in I})$ is the model companion $T^*_\cup$ of $T_\cup$, see~\cite{AldE}. In particular, $(\zz,+,(\elesub_p)_{p \in I})$ is an interpolative structure.

 \medskip\noindent \textbf{The algebraic closure of a finite field with a multiplicative circular orders.}
 Let $(\ff, +, \times)$ range over algebraic closures of finite fields, and let $\triangleleft$ range over multiplicatively invariant circular  orders on $\ff^\times$ (see \cite{Minh-1} for the definition).
 With $I =\{1,2\}$, $T_\cap $ the common theory of all $(\ff, \times)$, $T_1$ the common theory of all $(\ff, +, \times)$, and $T_2$ the common theory of all $(\ff, \times, \triangleleft)$, it follows from \cite{Minh-1} that $T_\cup$ has a model companion $T_\cup^*$, which is the common theory of all $(\ff, +, \times, \triangleleft)$. 
 In this case, $T_2$ is not model-complete, so the model-completeness assumption in Theorem~\ref{thm: first Theorem} is sufficient but not necessary.
 In fact, $T^*_\cup$ is the model companion of the theory of algebraically closed fields with a multiplicative circular order.
The initial motivation of this paper was to find a common generalization of this example, algebraically closed fields with independent valuations, and generic predicates.

 \medskip\noindent\textbf{Generic automorphisms.} Let $T_0$ be a one-sorted model-complete theory, and let $T$ be the theory whose models are $(\sM_0, \sigma)$, where $\sM_0 \models T_0$ and $\sigma$ is an automorphism of $\sM_0$.  If $T_0$ is $\text{ACF}$, then $T$ has a model companion $T^*$, which is called $\mathrm{ACFA}$~\cite{Cha-Hru}.
    In general, the question of existence of $T^*$ is subtle.

\medskip\noindent Suppose $\sM =(\sM_0, \sigma)$ is a model of $T$. Then $\sM$ is bi-interpretable with a two-sorted structure $(\sM_0, \sN_0; \iota_1, \iota_2)$, where  $\iota_1$ and $\iota_2$ are isomorphisms $\sM_0\to \sN_0$. In one direction, we can take $\sN_0 = \sM_0$, $\iota_1 = \text{id}$, and $\iota_2 = \sigma$. In the other direction, we can take $\sigma = \iota_1^{-1} \circ \iota_2$. Note also that $(\sM_0, \sN_0; \iota_1)$ and $(\sM_0, \sN_0; \iota_2)$ are both bi-interpretable with $\sM_0$.

 \medskip\noindent Let $I = \{1,2\}$, let $T_\cap$ be the two-sorted theory of two disjoint elementarily equivalent $T_0$-models, and let
$T_i$ be the two-sorted theory of two $T_0$-models with an isomorphism $\iota_i$ between them. Then $T_\cup$ is the theory of two $T_0$-models with two isomorphisms $\iota_1$ and $\iota_2$ between them. As described in the previous paragraph, $T$ and $T_\cup$ are existentially bi-interpretable. By Corollary~\ref{cor:delta1-cor}, if $T^*$ exists, then $T_\cup^*$ exists and is existentially bi-interpretable with $T^*$.  Furthermore, $T_1$ and $T_2$ are both existentially bi-interpretable with $T_0$. So, for example, $\mathrm{ACFA}$ is existentially  bi-interpretable with an interpolative fusion of two theories, each of which is existentially bi-interpretable with $\ACF$.

 \medskip\noindent \textbf{Differentially closed fields and free $\sD$-fields.} Let $T$ be the theory of differential fields of characteristic $0$ whose underlying field is algebraically closed. The theory $\text{DCF}_0$ of differentially closed fields is the model companion $T^*$ of $T$~\cite{RobinsonDCF}.

 \medskip\noindent Let $(K, \partial)$ be a model of $T$. 
Let $D = K[\varepsilon]/(\varepsilon^2)$ be the ring of dual numbers over $K$, and  $\pi : D \to K$ be the residue map $\pi(a + b\varepsilon) = a$. 
Then $(K, \partial)$ is existentially bi-interpretable with the two-sorted structure $ (K, D; \pi,\sigma_1,\sigma_2)$ where $\sigma_1: K \to D$ is given by $a \mapsto a+ 0\varepsilon$ and $\sigma_2: K \to D$ is given by $a \mapsto a+ \partial(a)\varepsilon$.
It can also be verified that $(K,D; \pi, \sigma_1)$ and $(K,D; \pi, \sigma_2)$ are isomorphic and mutually existentially interpretable with $K$. 

 \medskip\noindent Let $I =\{1,2\}$, $T_\cap = \text{Th}(K,D; \pi)$, $T_1 = \text{Th}(K,D; \pi, \sigma_1)$, and $T_2= \text{Th}(K,D; \pi, \sigma_2)$. 
Then $T_\cup$ is existentially bi-interpretable with $T$, so $T^*_\cup$ exists and is existentially bi-interpretable with $T^*= \text{DCF}_0$. 
Moreover, $T_2$ is a copy of $T_1$, and both are existentially bi-interpretable with $\mathrm{ACF}_0$. 
Thus,  $\mathrm{DCF}_0$ is existentially bi-interpretable with the interpolative fusion of two theories each of which is existentially bi-interpretable with $\mathrm{ACF}_0$.

\medskip\noindent $\mathrm{ACFA}$ and $\mathrm{DCF}_0$ admit a common generalization, the theories of existentially closed $\sD$-fields of Moosa and Scanlon~\cite{Moosa-Scanlon}. These theories are also essentially interpolative fusions; they will be discussed in future work.

\medskip\noindent \textbf{The random graph and related structures.} Let $T$ be the theory of infinite graphs with infinitely many edges. The theory of the random graph is the model companion $T^*$ of $T$. 
Let $S_V$ be the quotient $\{(v_1, v_2) \in V^2 : v_1 \neq v_2\}\slash{\sim}$, where the equivalence relation $\sim$ is defined by
$  (v_1,v_2) \sim (v_1', v_2')$ if and only if $\{v_1,v_2\} = \{v_1', v_2'\}.$ Let $\pi_V:  \{(v_1, v_2) \in V^2 : v_1 \neq v_2\} \to S_V$ be the quotient map seen as a relation on $V^2 \times S_V$,
and $E_V$ the image of $E$ under $\pi_V$, considered as a relation on $S_V$.
Then $(V, S_V; \pi_V, E_V)$ is existentially bi-interpretable with $(V;E)$.
    
 \medskip\noindent    Now suppose $I = \{1,2\}$,  $T_\cap$ is the common theory of $(V, S_V)$, $T_1$ is the common theory of $(V, S_V; \pi_V)$, and  $T_2$ is the common theory of $(V, S_V; E_V)$. 
    It can be checked that $T_1$ and $T_2$ are model-complete. Then $T_\cup$ is existentially bi-interpretable with $T$.
    Hence $T_\cup^*$ exists, and is existentially bi-interpretable $T^*$.
    It can also be shown that $T_1$ and $T_2$ are interpretable in the theory of infinite sets, so the theory of random graphs is up to existential bi-interpretation the interpolative fusion of two theories interpretable in the theory of infinite sets.
    Similar ideas also work with directed graphs, tournaments, $n$-hypergraphs, etc.

 \medskip\noindent \textbf{Generic Skolemizations.}
Suppose $L_0$ is a one-sorted language and $T_0$ is a model-complete $L_0$-theory.
Let $\varphi(x,y)$ be an $L_0$-formula with $y$ a single variable and $x$ a tuple of variables of length $n>0$, such that 
 $T_0 \models \forall x \exists^{\geq k} y\, \varphi(x,y) \text{ for all } k.$
Let $L= L_0\cup \{f\}$ with $f$ a new $n$-ary function symbol, and $T = T_0\cup \{\forall x\, \varphi(x,f(x))\}.$ Then $T$ is the ``$\varphi$-Skolemization'' of $T_0$. 
Winkler showed that $T$ has a model companion $T^*$, the ``generic $\varphi$-Skolemization'' of $T_0$, when $T_0$ eliminates $\exists^\infty$~\cite{Winkler}. Generic Skolemizations were also studied more recently in~\cite{KRExp}. 

\medskip\noindent Let $\sM = (\sM_0, f)$ range over the models of $T$. For each such $\sM$, let $E \subseteq M^{n+1}$ be defined in $\sM_0$ by $\varphi$, let $p_x: E \to M^n$ and $p_y: E \to M$ be the projection on the first $n$ coordinates and the last coordinate, respectively, and let $g: M^n \to E$ be the function $a \mapsto (a, f(a))$. Note that $p_x$ is an infinite-to-one surjection onto $M^n$, and $g$ is a section of $p_x$. Then the two-sorted structure $(\sM_0,E;p_x,p_y)$ is existentially bi-interpretable with $\sM_0$, and $(\sM_0,E;p_x,p_y,g)$ is existentially bi-interpretable with $\sM$.

\medskip\noindent Let $I=\{1, 2\}$, let $T_\cap$ the common theory of all $(M, E; p_x)$, let $T_1$ be the common theory of all $(\sM_0, E; p_x, p_y)$, and let $T_2$ be the common theory of all $(M,E; p_x, g)$. Then $T_1$ and $T_2$ are model-complete, and $T_\cup$ is existentially bi-interpretable with $T$. Hence $T_\cup^*$ exists and is bi-interpretable with $T^*$ when $T_0$ eliminates $\exists^\infty$. It can also be checked that $T_1$ is existentially bi-interpretable with $T_0$, and $T_2$ is interpretable in the theory of an infinite set.

\section{Pseudo-topological base}\label{sec: Pseudo-topological base}

\noindent In many examples, including most of those in Section~\ref{sec:examples}, the existence of the interpolative fusion can be explained by a simple idea: in the presence of a well-behaved notion of dimension on $L_\cap$-definable sets, a family $(X_i)_{i\in I}$ of $L_i$-definable sets is not separated (as defined in Section~\ref{sec:fusions}) if and only if they are simultaneously ``dense'' in some $L_\cap$-definable set. 

\medskip\noindent In this section, we provide a general framework abstracting these examples. Under very general hypotheses, we can use this framework to prove the existence of the interpolative fusion and provide an explicit $\forall\exists$-axiomatization. 

\medskip \noindent Throughout the rest of the paper, we introduce the following additional notational conventions: $L'$ is a first-order language extending $L$ with the same sorts as $L$,  $\sM'$ is an $L'$-structure and $\sM$ is its $L$-reduct, $T'$ is an $L'$-theory, and $T$ is the set of $L$-consequences of $T'$. 
Moreover, we assume the existence of a function $\dim$ that assigns an ordinal or the formal symbol $-\infty$ to each $\sM$-definable set whenever $\sM \models T$, so that for all $\sM$-definable $X,X_1,X_2\subseteq M^x$:
\begin{enumerate}
\item $\dim X_1 \cup X_2 = \max \{ \dim X_1, \dim X_2\}$;
\item $\dim X = - \infty$ if and only if $X = \emptyset$;
\item if $X$ is finite then $\dim X = 0$;
\item $\dim X=  \dim X( \sN)$ for any elementary extension $\sN$ of $\sM$.
\end{enumerate}
We call such a function $\dim$ an {\bf ordinal dimension} on $T$.
Examples include Morley rank on an $\aleph_0$-stable theory, U-rank on a  superstable theory, SU-rank on a supersimple theory, and o-minimal dimension on an o-minimal theory.
In fact, most natural examples of tame theories are equipped with a dimension, which is often canonical.

\medskip \noindent
Let $X$ be a definable subset of $M^x$ and $A$ an arbitrary subset of $M^x$.
Then $A$ is \textbf{pseudo-dense} in $X$ if $A$ intersects every nonempty definable $X' \subseteq X$ such that $\dim X' = \dim X$. We call $X$ a \textbf{pseudo-closure} of $A$ if $A \subseteq X$ and $A$ is pseudo-dense in $X$.
Lemma~\ref{lem: basicfacts} collects a few easy facts about pseudo-denseness, the proofs of which we leave to the readers.

\begin{lem}\label{lem: basicfacts}
Let $X$ and $X'$ be $\sM$-definable subsets of $M^x$, and let $A$ be an arbitrary subset of $M^x$. Then:
\begin{enumerate}
\item When $\dim X = 0$, $A$ is pseudo-dense in $X$ if and only if $X \subseteq A$.
\item If $A$ is pseudo-dense in $X$, $X' \subseteq X$, and $\dim X' = \dim X$, then $A$ is pseudo-dense in $X'$.
\item If $X^1, \ldots, X^n \subseteq X$ are $\sM$-definable, with $\dim X^i = \dim X$ for all $i$, and
$$ \dim X \setminus (X^1 \cup \ldots \cup X^n) < \dim X,$$
then $A$ is pseudo-dense in $X$ if and only if $A$ is pseudo-dense in each $X^i$.
\end{enumerate}
If in addition $X$ is a pseudo-closure of $A$, then:
\begin{enumerate}\setcounter{enumi}{3}
\item $A \subseteq X'$ implies $\dim X \leq \dim X'$.
\item If $X'$ is another pseudo-closure of $A$, then $\dim (X\triangle X') < \dim X = \dim X'$.  
\item If $A \subseteq X' \subseteq X$ then $X'$ is a pseudo-closure of $A$.
\end{enumerate}
\end{lem}

\noindent  We next introduce equivalent characterizations of interpolative structures under extra assumptions. Let $\sC$ be a collection of $\sM$-definable sets.
Then $X$ admits a \textbf{$\sC$-decomposition} if there is a finite family $(X_j)_{j \in J}$ from $\sC$ such that 
$$ \dim \Big( X \triangle \bigcup_{j \in J}X^j \Big) < \dim X, $$
and $X$ admits a {\bf $\sC$-patching} if there is a finite  family $(X^j, Y^j, f^j)_{j \in J}$ such that for all $j,j' \in J$:
\begin{enumerate}
\item $Y^j$ is in $\sC$.
\item $f^j: X^j \to Y^j$ is an $\sM$-definable bijection.
\item And finally,
$ \dim \Big( X \triangle \bigcup_{j \in J}X^j \Big) < \dim X. $
\end{enumerate}

\medskip\noindent We say that $\sC$ is a {\bf pseudo-cell collection} for $\sM$ if either every $\sM$-definable set admits a $\sC$-decomposition or $\dim$ is preserved under $\sM$-definable bijections and  every $\sM$-definable set admits a $\sC$-patching. Examples include the collection of irreducible varieties in an algebraically closed field and the collection of cells in an o-minimal structure.
 
\medskip \noindent Suppose $\dim$ is an ordinal dimension on $\text{Th}(\sM_\cap)$ and  $\sC$ is a collection of $\sM_\cap$-definable sets. 
We say $\sM_\cup$ is {\bf $\sC$-weakly interpolative}  if for all finite $J \subseteq I$, $X_\cap \in \sC$, and $( X_i)_{i \in J}$, where  $X_i$ is  $\sM_i$-definable and pseudo-dense in $X_\cap$, we have $\bigcap_{i \in J} X_i \neq \emptyset$.
If $\sC$ is the collection of all $\sM_\cap$-definable sets then we say that $\sM_\cup$ is \textbf{weakly interpolative}.
It is easy to see that if $\sC$ is a collection of pseudo-cells, then
$\sM_\cup$  is weakly interpolative if and only if $\sM_\cup$ is $\sC$-weakly interpolative.

\medskip \noindent Suppose $\sM'$ is an expansion of $\sM$.
Then $\sM'$ has \textbf{pseudo-closures in $\sM$} (with respect to the ordinal dimension $\dim$ on $\sM$) if every $\sM'$-definable set admits an $\sM$-definable pseudo-closure.
For later use, we say that $T'$ \textbf{has pseudo-closures in $T$} if every $T'$-model has pseudo-closures in its $L$-reduct.

\begin{prop}\label{prop:separation}
Suppose $J \subseteq I$ is finite and $X_i \subseteq M^x$ is $\sM_i$-definable for all $i \in J$. 
If there is an $\sM_\cap$-definable set $X$ in which each $X_i$ is pseudo-dense, then $( X_i)_{i \in J}$ is not separated. 
The converse implication holds provided $\sM_i$ has pseudo-closures in $\sM_\cap$ for all $i \in J$.
It follows that if $\sC$ is a collection of $\sM_\cap$-definable sets, then:
\begin{enumerate}
    \item If $\sM_\cup$ is interpolative, then $\sM_\cup$ is $\sC$-weakly interpolative. In particular, if  $\sM_\cup$ is   interpolative, then $\sM_\cup$ is weakly interpolative.
    \item Suppose moreover that each $\sM_i$ has pseudo-closures in $\sM_\cap$. If $\sM_\cup$ is weakly interpolative, or if $\sM_\cup$ is $\sC$-weakly interpolative and $\sC$ is a collection of pseudo-cells, then $\sM_\cup$ is interpolative.
\end{enumerate}
\end{prop}

\begin{proof}
We only prove the first two claims, since then (1) and (2) follow easily. 

For the first statement, suppose $X$ is a nonempty $\sM_\cap$-definable subset of $M^x$ in which each $X_i$ is pseudo-dense, and $(X^i)_{i \in J}$ is a family of $\sM_\cap$-definable sets satisfying  $X_i \subseteq X^i$ for each $i \in J$.
As $X_i$ is pseudo-dense in $X$ and disjoint from $X\setminus X^i$, we have $\dim X \setminus X^i < \dim X$ for all $i \in J$.
Hence, $$\dim \bigcup_{i \in J} (X \setminus X^i)  < \dim X.$$
Thus $\dim \bigcap_{i \in J}X^i \geq \dim X$, so $\bigcap_{i \in J}X^i$ is nonempty.

For the converse, assume $\sM_i$ has pseudo-closures in $\sM_\cap$ for each $i \in J$, and suppose $X_i$ is an $\sM_i$-definable set for each $i\in J$. We want to show that if there is no $\sM_\cap$-definable set $Z$ in which each $X_i$ is pseudo-dense, then $( X_i)_{i \in J}$ is separated. Simplifying notation, we assume $J = \{1,\ldots,n\}$. 
We show $(X_i)_{i = 1}^{n}$ is separated by applying simultaneous transfinite induction to $d_1,\ldots,d_n$  where $d_i$ is the dimension of any pseudo-closure of $X_i$.

Let $X^i$ be a pseudo-closure of $X_i$ for each $i$ and let
$$ Z = X^1 \cap \ldots \cap X^n .$$
If $\dim X^j = -\infty$ for some $j \in J$, then $X^j$ and $Z$ are both empty, so $(X^i)_{i = 1}^{n}$ separates $(X_i)_{i = 1}^{n}$.
If $\dim X^i = \dim Z$ for each $i$, then Lemma~\ref{lem: basicfacts}$(2)$ shows each $X_i$ is pseudo-dense in $Z$, contradiction.
After re-arranging the $X_i$ if necessary we suppose $\dim Z < \dim X^1$.
Let $Y_1 = X_1 \cap Z$.
As $(X_i)_{i = 1}^{n}$ cannot be simultaneously pseudo-dense in an $\sM_\cap$-definable set, it follows that $Y_1,X_2,\ldots,X_n$ cannot be simultaneously pseudo-dense in an $\sM_\cap$-definable set.
As the dimension of any pseudo-closure of $Y_1$ is strictly less then the dimension of $X^1$, an application of the inductive hypothesis provides $\sM_\cap$-definable sets $Y^1,\ldots,Y^n$ separating $Y_1,X_2,\ldots,X_n$.
It is easy to see 
$$ Y^1 \cup (X^1 \setminus Z),Y^2 \cap X^2,\ldots,Y^n \cap X^n$$
separates $X_1,\ldots,X_n$, which completes the proof.
\end{proof}

\noindent Let $T$ be an $L$-theory equipped with an ordinal dimension $\dim$ and $\sC$ an arbitrary collection of definable sets in $T$-models. We say that $\sC$ is a {\bf pseudo-cell collection} for $T$ if for all $\sM \models T$, $\sC \cap  \text{Def}(\sM)$ is a pseudo-cell collection for $\sM$.
\noindent We say that $T$ {\bf defines $\sC$-membership} if for every $L$-formula  $\varphi(x,y)$ there is an $L$-formula $\gamma(y)$ such that for all $\sM\models T$ and $b \in M^y$,
$$ \varphi(\sM,b) \text{ is in } \sC   \text{ if and only if }   \sM \models \gamma(b).  $$
We say that $T'$ {\bf defines pseudo-denseness over $\sC$} if for every $L'$-formula $\varphi'(x,y)$ and every $L$-formula $\varphi(x, z)$, there is an $L'$-formula $\delta'(y, z)$ such that if $\sM' \models T'$ and $c \in M^y$ with $\varphi(\sM',c) \in \sC$, then 
$$ \varphi'(\sM', b) \text{ is pseudo-dense in } \varphi(\sM', c)  \text{ if and only if }   \sM' \models \delta'(b, c).$$
If $\sC$ is the collection of all $\sM_\cap$-definable sets then we say that $T'$ {\bf defines pseudo-denseness} over $T$.

\medskip\noindent We say $T$ {\bf defines dimension} if for every ordinal $\alpha$, and every $L$-formula $\varphi(x,y)$, there is an $L$-formula $\delta_\alpha(y)$ such that for all $\sM \models T$ and $b \in M^y$
$$ \dim \varphi(\sM,b) = \alpha \quad   \text{if and only if}  \quad \sM \models \delta_\alpha(b).  $$
Note that if $T$ defines dimension then, by compactness, for every formula $\varphi(x,y)$, there are finitely many ordinals $\alpha_1, \ldots, \alpha_n$ such that for all  $\sM \models T$ and $b \in M^y$, we have $\dim(\varphi(\sM, b)) \in \{\alpha_1, \ldots, \alpha_n\} $.

\begin{prop} \label{Prop: Definingpdoverpc}
Suppose $\sC$ is a  pseudo-cell collection, $T$ defines $\sC$-membership and dimension, and $T'$ defines pseudo-denseness over $\sC$.
Then $T'$ defines pseudo-denseness over $T$.
\end{prop}

\begin{proof}
We only treat the case where $\dim$ is preserved under definable bijection and every definable set admits a $\sC$-patching.  
The other case is similar and easier.
Let $\varphi'(x,z')$ be an $L'$-formula and $\varphi(x,z)$ be an $L$-formula. 
We will produce an $L'$-formula $\delta'(z, z')$ such that whenever $\sM' \models T'$ and $\sM = \sM' \res L$, we have 
$$ \sM' \models \delta(c,c') \text{ if and only if } \varphi'(\sM',c') \text{ is pseudo dense in } \varphi(\sM,c).$$

Applying compactness we obtain $n$ and an $L$-formula $\psi(x,y,w)$ such that for all $\sM' \models T'$ and $\sM = \sM' \res L$  we have the following
\begin{enumerate}
    \item For all $d \in M^w$, the formula $\psi(x,y,d)$ defines a function $f_d$ from a subset $X_d$ of $M^x$ to a subset $Y_d$ of $M^y$.
    \item For each $c \in M^z$, there exists $J \subseteq \{1, \ldots, n\}$ and $(d_j)_{j \in J}$ from $M^w$, such that
    $(X_{d_j}, Y_{d_j}, f_{d_j})_{j \in J}$ as defined in (1)
    is a $\sC$-patching of $\varphi(\sM, c)$.
\end{enumerate}
Now we have that $\varphi'(\sM, c')$ is pseudo-dense in $\varphi(\sM, c)$ if and only if there are $J$, and $(d_j)_{j \in J}$ as in (1) and (2) such that $f_{d_j}( \varphi'(\sM, c') \cap  X_{d_j} )$ is pseudo-dense in $Y_{d_j}$ for  all $j\in J$ such that  $\dim X_{d_j} = \dim \varphi(\sM, c)$.
From the analysis above, it is easy to see that we can choose the desired formula $\delta'(z, z')$.   
\end{proof}

\noindent With all the pieces in place, we can prove the main theorem of this section, which has Theorem~\ref{thm:mainsec3} from the introduction as a special case.

\begin{thm} \label{thm:approxinterp}
Suppose $\sC$ is a collection of definable sets of $T_\cap$-models  such that $T_\cap$ defines  $\sC$-membership, and each $T_i$ defines pseudo-denseness over $\sC$. 
Then we have the following:
\begin{enumerate}
    \item The class of $\sC$-weakly interpolative $T_\cup$-models is elementary.
    \item If $\sC$ is a pseudo-cell collection for $T_{\cap}$, then the class of weakly interpolative $T_\cup$-models is elementary.
    \item If, in addition, each $T_i$ has pseudo-closures in $T_\cap$, then the interpolative fusion $T^*_\cup$ exists.
\end{enumerate}
In particular, taking $\sC$ to be the collection of all definable sets in $T_\cap$-models:
\begin{enumerate}
    \item[(4)] If $T_i$ defines pseudo-denseness over $T_\cap$ for all $i\in I$, then the class of weakly interpolative $T_\cup$-models is elementary.
    \item[(5)] If, in addition, $T_i$ has pseudo-closures in $T_\cap$ for all $i\in I$, then $T^*_\cup$ exists.
\end{enumerate}
\end{thm}

\begin{proof}
It suffices to prove (1). Then (2) follows from (1) and the fact, noted above, that $\sC$ is a   pseudo-cell collection for $T_\cap$, then
$\sM_\cup$  is weakly interpolative if and only if $\sM_\cup$ is $\sC$-weakly interpolative. Now (3) follows from Proposition~\ref{prop:separation}. Finally, (4) and (5) are special cases of (2) and (3), respectively. 

Let $\varphi_\cap(x,y)$ be an $L_\cap$-formula, let $J\subseteq I$ be finite, and let $\varphi_i(x,z_i)$ be an $L_i$-formula for each $i\in J$. Let $\gamma_\cap(y)$ be an $L_\cap$-formula defining $\sC$-membership for $\varphi_\cap(x,y)$, and let $\delta_i(y,z_i)$ be an $L_i$-formula defining pseudo-denseness over $\sC$ for $\varphi_\cap(x,y)$ and $\varphi_i(x,z_i)$ for each $i \in J$. For simplicity, we assume $J = \{1,\dots,n\}$. Then we have the following axiom: \[\forall y,z_1,\dots,z_n \left(\left(\gamma_\cap(y) \wedge \bigwedge_{i=1}^n \delta_i(y,z_i)\right)\rightarrow \exists x\, \bigwedge_{i=1}^n \varphi_i(x,z_i)\right).\]
Then $T_\cup$, together with one axiom of the above form for each choice of $\varphi_\cap(x,y)$, $J$,  and $\varphi_i(x,z_i)$ for $i \in J$ as above, axiomatizes the class of $\sC$-weakly interpolative $T_\cup$-models. 
\end{proof}

\noindent We refer to the axioms obtained in the proof of Theorem~\ref{thm:approxinterp} as the \textbf{pseudo-topological axioms} for $T^*_\cup$.

\medskip \noindent At present, we know that Theorem~\ref{thm:approxinterp} applies to all examples described in Example~\ref{ex:exampleblock1} and \ref{ex:exampleblock2} except for the integers with $p$-adic valuations and differentially closed fields.
We do not know if Theorem~\ref{thm:approxinterp} applies to differentially closed fields, as in this case we do not have a good understanding of definable sets in the base theory.

\medskip \noindent
We now show that Theorem~\ref{thm:approxinterp} does not apply to the example of the integers with $p$-adic valuations. 
Let $\dim$ be the canonical ordinal dimension on the additive group of integers, which coincides with U-rank, acl-dimension, etc., see   \cite{Conant}.

\begin{prop}\label{prop: approximable not necessary}
Suppose $(Z;+,\times)$ is an $\aleph_1$-saturated elementary extension of $(\mathbb{Z};+,\times)$.
Fix a prime $p$ and let $v_p$ be the $p$-adic valuation on $Z$. 
Given $a,b \in Z$ we declare $a \elesub_p b$ if and only if $v_p(a) \leq v_p(b)$.
Fix $N$ in $Z$ such that $v_p(N) \geq n$ for all $n \in \mathbb{N}$ and let $$E = \{z\in Z\mid N\elesub_p z\}.$$
Then $E$ does not have a pseudo-closure in $(Z;+)$.
So $(Z;+,\elesub_p)$ does not have pseudo-closures in $(Z;+)$.
\end{prop}

\begin{proof}
The quantifier elimination for $\text{Th}(\mathbb{Z};+)$ implies that every $(Z;+)$-definable subset of $Z$ is a finite union of sets of the form $(kZ + l) \setminus F$ for $k,l \in \mathbb{Z}$ and finite $F$.
So if $E$ has a pseudo-closure in $(Z;+)$, then $E$ is pseudo-dense in $kZ + l$ for some $k,l \in \mathbb{Z}$ with $k\neq 0$.

So we fix $k,l \in \mathbb{Z}$ with $k\neq 0$ and show that $E$ is not pseudo-dense in $k Z + l$.
As $\dim(kZ+l) = 1$, 
it is enough to show that $E$ is disjoint from some infinite definable subset of $kZ + l$.

If $l \neq 0$, let $n = v_p(l)$, so $l = p^nm$ for some $m\in \mathbb{Z}$ which is coprime to $p$. Then $p^{n+1}kZ + l = p^n(pkZ + m)$ is an infinite definable subset of $kZ + l$. Every element of this set has $p$-adic valuation at most $n\in \mathbb{N}$, so it is disjoint from $E$.

If $l = 0$, consider $kZ \setminus pk Z$. This is an infinite definable subset of $kZ$, and if $a \in kZ \setminus pk Z$, then $v_p(a) = v_p(k) \in \mathbb{N}$.
So $E$ is disjoint from $kZ \setminus pk Z$.
\end{proof}

\begin{rem}
One can apply the ``quasi-coset'' decompositions given in \cite[Theorem 4.10]{Conant} to show that 
$ \{ (k,l) \in \mathbb{Z}^2 : k \elesub_p l \} $
does not have a pseudo-closure in $\mathbb{Z}^2$.
This presents some technical difficulties, so we do not include it here.
As every $(\mathbb{Z};+,\elesub_p)$-definable subset of $\mathbb{Z}$ is $(\mathbb{Z};+)$-definable~\cite{AldE}, we must pass to an elementary extension to obtain a unary set without a pseudo-closure.
\end{rem}

\section{Tame topological base}\label{section-tame-topology}
\noindent We consider in this section specializations of Theorem~\ref{thm:approxinterp} to settings in which the base theory $T_\cap$ admits a well-behaved definable topology. Throughout, we maintain the notational conventions described at the beginning of Section~\ref{sec: Pseudo-topological base}, and we explore the degree to which the pseudo-topological notions defined there agree with natural topological notions.
We show, among other things, that if $\sM$ is o-minimal, then $\sM'$ has pseudo-closures in $\sM$ if and only if the closure of every $\sM'$-definable set is $\sM$-definable.
This equivalence only depends on two well-known facts from o-minimality.
One of these is known as the frontier inequality, and we refer to the other as the residue inequality. Whenever these inequalities hold in our abstract setting, we automatically obtain definability of pseudo-denseness. 

\medskip\noindent A {\bf definable topology} $\sT$ on $\sM$ consists of a topology $\sT_x$ on each $M^x$, such that $\{ \varphi(\sM,b) : b \in M^y \}$ is a basis for $\sT_x$, for some $L$-formula $\varphi(x,y)$.
Note that we also obtain a definable topology on every model of $\Th(\sM)$. For the rest of Section~\ref{section-tame-topology}, we suppose $\sT$ is a definable topology on $\sM$ and $\dim$ is an ordinal dimension on $T = \Th(\sM)$ such that $T$ defines dimension. 

\medskip\noindent 
 Let $A$ be a subset of $M^x$. We denote by $\cl(A)$ the closure of $A$ with respect to $\sT_x$.
The {\bf frontier} of $A$, $\text{fr}(A)$, is defined as $\cl(A) \setminus A$.
Since $\sT$ is a definable topology, the interior, closure, and frontier of a definable subset of $M^x$ are all definable. 

\medskip \noindent In general there need be no connection between pseudo-denseness and $\sT$-denseness.
We give conditions under which the two naturally relate.
We say $\sM$ satisfies the \textbf{frontier inequality} if $$\dim \text{fr}(X) < \dim X \quad \text{for all definable } X.$$
This is a strong assumption, which in particular implies, by a straight-forward induction on dimension, that every definable set is a Boolean combination of open definable sets.

\begin{lem}\label{lem:topo}
Suppose $\sM$ satisfies the frontier inequality and $X' \subseteq X$ are $\sM$-definable sets.
If $\dim X' = \dim X$, then $X'$ has nonempty interior in $X$. 
\end{lem}

\begin{proof}
If $X'$ has empty interior in $X$, then $X\setminus X'$ is dense in $X$, and so $X'\subseteq X\subseteq \cl(X \setminus X')$. In particular, $X'\subseteq \text{fr}(X\setminus X')$. 
The frontier inequality implies $\dim X' < \dim X \setminus X' \leq \dim X$.
\end{proof}

\begin{lem}\label{lem:topo0}
The following are equivalent:
\begin{enumerate}
    \item $\sM$ satisfies the frontier inequality.
    \item If $A \subseteq M^x$ is dense in a definable $X \subseteq M^x$ then $A$ is pseudo-dense in $X$.
\end{enumerate}
\end{lem}

\begin{proof}
Suppose that $\sM$ satisfies the frontier inequality and that $A \subseteq M^x$ is dense in a definable set $X \subseteq M^x$.
Suppose $X' \subseteq X$ is definable and $\dim X' = \dim X$.
Lemma~\ref{lem:topo} implies that $X'$ has nonempty interior in $X$.
Thus $A$ intersects $X'$.
It follows that $A$ is pseudo-dense in $X$.

Conversely, assume (2), and let $X\subseteq M^x$ be definable. Since $X$ is dense in $\cl(X)$, $X$ is also pseudo-dense in $\cl(X)$. 
As $X \cap \text{fr}(X) = \emptyset$, we have $\dim \text{fr}(X) < \dim \cl(X)$.
It follows that $\dim X = \dim\cl(X)$, so the frontier inequality holds. 
\end{proof}

\noindent Pseudo-density does not, in general, imply density. 
For example, if $X \subseteq M^x$ is an infinite definable set and $p \in M^x$ does not lie in $\cl(X)$, then $X$ is pseudo-dense in $X \cup \{ p \}$ but not dense in $X \cup \{p\}$.
However, the converse to (2) does hold for certain definable sets, which we call pure dimensional.

\medskip \noindent Let $X \subseteq M^x$ be definable.
Given $p \in X$, we define
$$ \dim_p X = \min \{ \dim (U \cap X) : 
U \text{ is a definable neighborhood of } p \} .$$
We say that $X$ is \textbf{pure dimensional} if $\dim_p X = \dim X$ for all $p \in X$. Equivalently, $X$ is pure dimensional if and only if $\dim U = \dim X$ for all nonempty definable open subsets of $X$.

\begin{lem}\label{lem:topo1}
Suppose $X \subseteq M^x$ is definable.
Then the following are equivalent:
\begin{enumerate}
    \item $X$ is pure dimensional.
    \item If a subset $A$ of $M^x$ is pseudo-dense in $X$, then $A$ is dense in $X$.
\end{enumerate}
\end{lem}

\begin{proof}
Suppose $X$ is not pure dimensional.
Let $U$ be a definable nonempty open subset of $X$ such that $\dim U < \dim X$.
Then $X \setminus U$ is pseudo-dense in $X$ and not dense in $X$.

Suppose $X$ is pure dimensional and $A$ is pseudo-dense in $X$.
Suppose $U$ is a nonempty open subset of $X$.
Then there is a definable nonempty open subset $U'$ of $U$.
Then $\dim U' = \dim X$, so $A$ intersects $U'$.
Hence $A$ is dense in $X$.
\end{proof}

\noindent Proposition~\ref{prop:dim-pure-indecom} gives another characterization of pure dimensional sets. 
We will not use this characterization, so we leave its proof to the reader. 

\begin{prop}
\label{prop:dim-pure-indecom}
Suppose $X\subseteq M^x$ is definable. If $X$ is  pure dimensional, then there are no definable sets $X^1$ and $X^2$ such that $X = X^1 \cup X^2$,  $X^1$ and $X^2$ are closed in $X$, neither $X^1$ nor $X^2$ contains the other, and $\dim X^1 \neq \dim X^2$. If $\sM$ satisfies the frontier inequality, then the converse holds.
\end{prop}

\noindent 
For a definable $X\subseteq M^x$, we define the \textbf{essence} of $X$, $\core(X)$, and the \textbf{residue} of $X$, $\rs(X)$:
\begin{align*}
\core(X) &= \{p \in X : \dim_p X = \dim X\}\\
\rs(X) &= \{ p \in X : \dim_p X < \dim X \}
\end{align*}
As $\sT_x$ admits a definable basis, and $T$ defines dimension, it follows that $\core(X)$ and $\rs(X)$ are definable.

\medskip\noindent We say that $\sM$ satisfies the \textbf{residue inequality} if
$$ \dim \rs(X) < \dim X \quad \text{for all definable } X. $$

\begin{lem}\label{lem:core}
If $\sM$ satisfies the residue inequality, then for all definable $X\subseteq M^x$, $\core(X)$ is pure dimensional. 
\end{lem}
\begin{proof}
Let $p\in \core(X)$, and let $U$ be a definable neighborhood of $p$. We need to show that $\dim(U\cap \core(X)) = \dim \core(X)$. By the residue inequality, $\dim \rs(X) < \dim X$. So  $\dim \core(X) = \dim X = \dim(U\cap X)$ because $p\in \core(X)$. Thus, \[\dim \core(X) = \dim(U\cap X)= \dim((U\cap X)\setminus \rs(X)) =\dim
(U\cap \core(X)).\qedhere\]
\end{proof}

\noindent We will not use Proposition~\ref{prop:residue}, but we include it here, since it provides additional motivation for the residue inequality. 

\begin{prop}\label{prop:residue}
$\sM$ satisfies the residue inequality if and only if every definable set is a finite disjoint union of pure dimensional definable sets. 
\end{prop}
\begin{proof}
Suppose first that $\sM$ satisfies the residue inequality. Let $X\subseteq M^x$ be definable. We argue by induction on $\dim X$. If $\dim X = -\infty$, then $X = \emptyset$ and the conclusion holds vacuously. Otherwise, $X$ is the disjoint union of $\core(X)$ and $\rs(X)$. By Lemma~\ref{lem:core}, $\core(X)$ is pure dimensional, and by the residue inequality $\dim \rs(X) < \dim X$, so by induction $\rs(X)$ is a finite disjoint union of pure dimensional definable sets. 

Conversely, for any definable set $X$, suppose that $X$ is a disjoint union of pure dimensional definable sets $Y_1,\ldots, Y_m$. We will show that $\dim \rs(X) < \dim X$. We may assume without loss of generality that $1 \leq j \leq m$ is such that $$\dim Y_k = \dim X \text{ when } k \leq j \quad \text{ and} \quad \dim Y_k < \dim X  \text{ when } k > j.$$ 
Let $p\in \rs(X)$, and suppose for contradiction that  $p\in Y_k$ for some $k\leq j$. Then since $Y_k$ is pure dimensional, $\dim_p Y_k = \dim Y_k = \dim X$, so for any definable neighborhood $U$ of $p$, $$\dim X = \dim (U\cap Y_k) \leq \dim (U\cap X)\leq \dim X.$$ So $\dim_p X = \dim X$, contradicting the fact that $p\in \rs(X)$. Thus $\rs(X) \subseteq \bigcup_{k > j} Y_k$, and $\dim \rs(X) \leq \dim \bigcup_{k > j} Y_k < \dim X$. 
\end{proof}

\noindent We say $\sT$ is \textbf{$\dim$-compatible}  if  $\sM$ satisfies both the frontier inequality and the residue inequality.
Definability of the dimension and the topology ensure that $\dim$-compatibility is an elementary property, i.e., the topology on any model of $T$ is $\dim$-compatible.

\begin{prop}\label{prop:topo-2}
 Suppose $\sT$ is $\dim$-compatible. Suppose $X \subseteq M^x$ is definable and $A \subseteq M^x$. 
Then $A$ is pseudo-dense in $X$ if and only if $A$ is dense in $\core(X)$.
\end{prop}

\begin{proof}
Since $\dim \rs(X) < \dim X$ and $\dim \core(X) = \dim X$, $A$ is pseudo-dense in $X$ if and only if $A$ is pseudo-dense in $\core(X)$. The equivalence then follows from Lemma~\ref{lem:topo0}, Lemma~\ref{lem:topo1}, and Lemma~\ref{lem:core}.
\end{proof}

\begin{prop}\label{prop:top-def-ps}
Suppose $\sT$ is $\dim$-compatible. Any expansion $T'$ of $T$ defines pseudo-denseness over $T$.
\end{prop}

\begin{proof}
Suppose $\sM$ is a $T$-model and $\sM'$ is a $T'$-model expanding $\sM$.
Suppose $(X_b)_{b \in M^y}$ and $(X'_c)_{c \in M^z}$ are families of subsets of $M^x$, which are $\sM$-definable and $\sM'$-definable, respectively. By Proposition~\ref{prop:topo-2}, $X'_c$ is pseudo-dense in $X_b$ if and only if $X'_c$ is dense in $\core(X_b)$. 

Using definability of the topology and dimension, essences of definable sets are uniformly definable, i.e., there is an $\sM$-definable family $(Y_b)_{b\in M^y}$ such that $Y_b = \core(X_b)$ for all $b\in M^y$. Thus $X'_c$ is pseudo-dense in $X_b$ if and only if $X'_c$ is dense in $Y_b$. And using definability of the topology, the set of all $(b,c)$ such that $X'_c$ is dense in $Y^b$ is definable.
\end{proof}

\begin{prop}\label{prop:topo-3}
Suppose $\sT$ is $\dim$-compatible. Suppose $\sM'$ expands $\sM$. 
Then $\sM'$ has pseudo-closures in $\sM$ if and only if the closure of any $\sM'$-definable set is $\sM$-definable.
\end{prop}

\begin{proof}
Suppose that the closure of any $\sM'$-definable set is $\sM$-definable. Then for any $\sM'$-definable $X\subseteq M^x$, $\cl(X)$ is a pseudo-closure of $X$ by Lemma~\ref{lem:topo0}.

Conversely, suppose $\sM'$ has pseudo-closures in $\sM$ and $X' \subseteq M^x$ is $\sM'$-definable. Let $X$ be a pseudo-closure of $X'$. We apply induction to the dimension of $X$. If $\dim X = -\infty$, then $X'$ is empty and trivially $\sM$-definable.
Now suppose $\dim X \geq 0$.
We have
$$ \cl(X') =  \cl(X'\cap \core(X)) \cup \cl(X'\cap \rs(X)).$$
Since $X'$ is pseudo-dense in $X$, $X'$ is dense in $\core(X)$ by Proposition~\ref{prop:topo-2}. It follows that $\cl(X'\cap \core(X)) = \cl(\core(X))$, which is $\sM$-definable. 
As $(X'\cap \rs(X)) \subseteq \rs(X)$, any pseudo-closure of $(X' \cap \rs(X))$ has dimension at most $\dim \rs(X) < \dim X$. So $\cl(X'\cap \rs(X))$ is $\sM$-definable by induction. Thus $\cl(X')$ is a union of two $\sM$-definable sets and is therefore $\sM$-definable.
\end{proof}

\noindent We conclude this section by giving examples of structures with $\dim$-compatible definable topologies.
In those examples, $\sT$ and $\dim$ are canonical, so we do not describe them in detail.
In each case, existence of pure dimensional decompositions (and hence the residue inequality, by Proposition~\ref{prop:residue}) follows from the appropriate cell decomposition or ``weak cell decomposition''.
In different settings, cells (or ``weak cells'') have different definitions, but they are easily seen to be pure dimensional in each case.

\medskip \noindent The most familiar case is when $\sM$ is an o-minimal expansion of a dense linear order, see \cite{lou-book}.
Similarly, it follows from \cite[Propositions 4.1 and 4.3]{SW-tametop} that if $\sM$ is a dp-minimal expansion of a divisible ordered abelian group, then the usual order topology is compatible with dp-rank (and dp-rank agrees with several other natural notions of dimension~\cite[Proposition 2.4]{SW-tametop}).
This covers the case when $\sM$ is an expansion of an ordered abelian group with weakly o-minimal theory.
It is shown in Johnson's thesis \cite{Johnson-thesis} that if $\sM$ is a dp-minimal, but not strongly minimal, expansion of a field, then $\sM$ admits a canonical non-discrete definable field topology. Taking the product topology on $M^n$ for all $n$, it is shown in \cite{SW-tametop} that this topology is compatible with dp-rank.
This covers the case of a C-minimal expansion of an algebraically closed field or a P-minimal expansion of a $p$-adically closed field.
It was previously shown in \cite{KDL-pmin} that P-minimal expansions of $p$-adically closed fields satisfy the frontier inequality and admit pure dimensional decompositions.

\medskip \noindent We say that $T$ is an \textbf{open core} of $T'$ if the closure of every $T'$-definable set in every $T'$-model $\sM'$ is $\sM = \sM' | L$ definable.
Propositions~\ref{prop:top-def-ps} and \ref{prop:topo-3} together yield Theorem~\ref{thm:psd4}, which generalizes Theorem~\ref{thm: mainsection5} from the introduction.

\begin{thm}\label{thm:psd4}
If $T_\cap$ admits a definable ordinal dimension $\dim$ and a 
$\dim$-compatible definable topology, and $T_\cap$ is an open core of each $T_i$, then $T^{*}_\cup$ exists.
In particular, if $T_\cap$ is an o-minimal expansion of a dense linear order or a P-minimal expansion of a $p$-adically closed field, and $T_\cap$ is an open core of each $T_i$, then $T^{*}_\cup$ exists. 
\end{thm}

\noindent We give an application of Theorem~\ref{thm:psd4}. 
Suppose $T_\cap$ is a complete and model-complete o-minimal theory extending the theory of ordered abelian groups.
For each $i \in I$, let $T_i$ be the theory of a $T$-model $\sN$ equipped with a unary predicate $R_i$ defining a dense elementary substructure of $\sN$.
Then $T_i$ is model-complete by \cite[Thm 1]{dense-pairs} and $T_\cap$ is an open core of $T_i$ \cite[Section 5]{DMS}.
Applying Theorem~\ref{thm:psd4}, we see that the theory $T_\cup$ of a $T$-model $\sN$ equipped with a family $(R_i)_{i \in I}$ of unary predicates defining dense elementary substructures of $\sN$ has a model companion.

\section{$\aleph_0$-stable base}\label{section:aleph-base}
\noindent In almost all of the examples from Section~\ref{sec:examples}, the base theory $T_\cap$ is $\aleph_0$-stable (in fact, $T_\cap$ is almost always interpretable in $\ACF$ or the theory of an infinite set). In this section, we specialize Theorem~\ref{thm:approxinterp} to this setting, where the natural ordinal dimension is Morley rank, and we obtain pseudo-closures for free. 

\noindent We keep the additional notational conventions described at the beginning of Section~\ref{sec: Pseudo-topological base}. Throughout this section, $T$ is $\aleph_0$-stable, $\dim$ is Morley rank, and $\mult$ is Morley degree.

\medskip \noindent Suppose $X^1$ and $X^2$ are $\sM$-definable subsets of $M^x$.
We will say that $X^1$ is {\bf almost a subset} of $X^2$ and write $X^1 \subseteq_{\text{a}} X^2$ if
$$    \dim( X^1 \setminus X^2) < \dim (X^1).$$
We will say that  $X^1$ is {\bf almost equal} to $X^2$ and write $X^1 =_{\text{a}} X^2$ if $X^1\subseteq_{\text{a}} X^2$ and $X^2 \subseteq_{\text{a}} X^1$. It is easy to see that $=_\text{a}$ is an equivalence relation.
An $\sM$-definable subset $X$ of $M^x$ is \textbf{almost  irreducible} if whenever $X = X^1 \cup X^2$ for $\sM$-definable $X^1$ and $X^2$, we have $X =_{\text{a}} X^1$ or $X =_{\text{a}} X^2$.
Any $\sM$-definable set of Morley degree one is almost irreducible, and the converse holds when $\text{Th}(\sM)$ defines 
Morley rank or when $\sM$ is $\aleph_0$-saturated.

\medskip\noindent Proposition~\ref{prop:induced-1} is the main advantage of assuming that $T_\cap$ is $\aleph_0$-stable in our setting.

\begin{prop}\label{prop:induced-1}
Let $\sM\models T$. Every $A \subseteq M^x$ has a pseudo-closure. More precisely, an $\sM$-definable set $X \subseteq M^x$ is a pseudo-closure of $A$ if and only if $A \subseteq X$ and for all $\sM$-definable  $X' \subseteq M^x$ with $A \subseteq X'$,
$$(\dim X, \mult X) \leq_{\mathrm{Lex}} (\dim X', \mult X').$$ 
It follows that every expansion of $T$ has pseudo-closures in $T$.
\end{prop}
\begin{proof}
By standard properties of Morley rank and degree in $\aleph_0$-stable theories, for any $\sM$-definable $X$ and $X'$, if $(\dim X',\mult X')<_{\mathrm{Lex}} (\dim X,\mult X)$, then $\dim (X\setminus X') = \dim X$. If $X'\subseteq X$, then the converse is true. 
 
Let $X$ be a pseudo-closure of $A$, so $A\subseteq X$, and suppose for contradiction that there is some $\sM$-definable $X'\subseteq M^x$ with $A\subseteq X'$ and $(\dim X',\mult X') <_{\mathrm{Lex}} (\dim X, \mult X)$. Then $\dim(X\setminus X') = \dim X$, but $A\cap (X\setminus X') = \emptyset$, contradicting the fact that $A$ is pseudo-dense in $X$. 

Conversely, suppose $A\subseteq X$ and $(\dim X, \mult X)$ is minimal in the lexicographic order among $\sM$-definable sets containing $A$. Then for any $\sM$-definable $X'\subseteq X$ with $\dim X' = \dim X$, $(\dim (X\setminus X'),\mult (X\setminus X')) <_{\mathrm{Lex}} (\dim X,\mult X)$. It follows that $A\not\subseteq (X\setminus X')$, so $A\cap X'\neq \emptyset$. Hence $X$ is a pseudo-closure of $A$.
\end{proof}

\noindent Corollary~\ref{cor:omega-stable-pseudo 1} now follows immediately from Proposition~\ref{prop:separation}.

\begin{cor}\label{cor:omega-stable-pseudo 1}
If $\Th(\sM_\cap)$ is $\aleph_0$-stable and $\dim$ is Morley rank, then $\sM_\cup$ is interpolative if and only it is weakly interpolative.
\end{cor}

\noindent In Proposition~\ref{prop: approximable not necessary}, we gave a  concrete example of an expansion of $T = \Th(\mathbb{Z};+)$ that does not have pseudo-closures in $T$.
It is well known that $T$ is superstable but not $\aleph_0$-stable, so this demonstrates that superstability is not sufficient for Proposition~\ref{prop:induced-1}. For the reader who is still looking for a free ride outside of the $\aleph_0$-stable context, Proposition~\ref{prop:omeg-approx} will dash this hope. 

\medskip\noindent  If $\dim_1, \dim_2$ are ordinal dimensions on an $L^\diamondsuit$-theory $T^\diamondsuit$ then we say $\dim_1$ is \textbf{smaller than} $\dim_2$ if $\dim_1 X \leq \dim_2 X$ for all definable sets $X$.

\begin{rem}\label{rem:RM1}
The theory $T^\diamondsuit$ is $\aleph_0$-stable if and only if it admits an ordinal dimension $\dim$ such that for every $T^\diamondsuit$-model $\sM^\diamondsuit$, $\sM^\diamondsuit$-definable set $X$, and family $(X_n)_{n \in \mathbb{N}}$ of pairwise disjoint $\sM^\diamondsuit$-definable subsets of $X$, we have $\dim X_n < \dim X$ for some $n$.
If $T^\diamondsuit$ is $\aleph_0$-stable, then Morley rank is the smallest ordinal dimension with this property. 
\end{rem}

\begin{prop}\label{prop:omeg-approx}
Suppose $L^\diamondsuit$ is countable and $\dim^\diamondsuit$ is an ordinal dimension on a complete $L^\diamondsuit$-theory $T^\diamondsuit$. If $T^\diamondsuit$ is not $\aleph_0$-stable, then there is an expansion of $T^\diamondsuit$ that does not have pseudo-closures in $T^\diamondsuit$.
\end{prop}

\begin{proof}
Suppose $T^\diamondsuit$ is not $\aleph_0$-stable.
Applying Remark~\ref{rem:RM1}, we obtain a $T^\diamondsuit$-model $\sM^\diamondsuit$, an $\sM^\diamondsuit$-definable set $X$ with $\dim^\diamondsuit X = \alpha$, and a sequence  $(X_n)_{n \in \mathbb{N}}$ of pairwise disjoint $\sM^\diamondsuit$-definable  subsets of $X$ such that $\dim^\diamondsuit X_n = \alpha$ for all $n$. Since $X$ and each $X_n$ are definable with parameters from a countable elementary submodel, we may assume $\sM^\diamondsuit$ is countable. 

Given $S \subseteq \mathbb{N}$, let $A_S = \bigcup_{n \in S} X_n$.
We show that $A_S$ does not have a pseudo-closure for uncountably many $S \subseteq \mathbb{N}$.
Suppose $S \subseteq \mathbb{N}$ is nonempty and $X'$ is a pseudo-closure of $A_S$.
As $A_S \subseteq X$, we have $\dim^\diamondsuit X' \leq \alpha$.
As $S$ is nonempty, we have $X_n \subseteq X'$ for some $n$, so $\dim^\diamondsuit X' \geq \alpha$.
Thus any pseudo-closure $X'$ of $A_S$ has $\dim^\diamondsuit X' = \alpha$.

Now suppose $S,S' \subseteq \mathbb{N}$ are nonempty and $S \not\subseteq S'$.
We show any pseudo-closure of $A_{S}$ is not a pseudo-closure of $A_{S'}$.
Fix $n \in S \setminus S'$ and suppose $X'$ is a pseudo-closure of $A_S$.
Then $\dim^\diamondsuit X' = \alpha$, $X_n$ is an $\sM^\diamondsuit$-definable subset of $X'$ with $\dim^\diamondsuit X_n = \alpha$,  but $X_n$ is disjoint from $A_{S'}$.
Thus $X'$ is not a pseudo-closure of $A_{S'}$.

Let $\mathfrak{J}$ be an uncountable collection of nonempty subsets of $\mathbb{N}$ such that $S \not\subseteq S'$ for all distinct $S,S' \in \mathfrak{J}$.
If $S,S' \in \mathfrak{J}$ are distinct, then $A_S$ and $A_{S'}$ cannot have a common pseudo-closure.
As $\sM^\diamondsuit$ and $L$ are countable, there are only countably many $\sM^\diamondsuit$-definable sets,  so there are uncountably many $S \in \mathfrak{J}$ such that $A_S$ does not have a pseudo-closure.
The expansion of $\sM^\diamondsuit$ by a predicate defining any such $A_S$ does not have pseudo-closures in $\sM^\diamondsuit$.
It follows that the theory of this expansion does not have pseudo-closures in $T^\diamondsuit$.
\end{proof}

\noindent We next give a useful characterization of definability of pseudo-denseness over an $\aleph_0$-stable theory. Proposition \ref{prop:induced-1} motivates the following definition. 
Suppose $M'$ is a model of $T'$, $\sM  = \sM' | L$, and $X' \subseteq M^x$ is $\sM'$-definable. Define
 $$ \dim' X' = \dim X \quad \text{and}\quad \mult' X' = \mult X $$
 where $X$ is any pseudo-closure of $X'$.
Lemma~\ref{lem:induced-2} is an immediate consequence of Proposition \ref{prop:induced-1}.
 
\begin{lem}\label{lem:induced-2}
For $A \subseteq M^x$ and $\sM$-definable $X \subseteq M^x$, we have the following:
\begin{enumerate}
   \item $A$ is pseudo-dense in $X$ if and only if we have both $\dim'( X \cap A) = \dim(X)$ and $\mult'(X \cap A) = \mult(X)$.
\item If $X$ is almost irreducible, then $A$ is pseudo-dense in $X$ if and only if $\dim'( X \cap A) = \dim(X)$.
\end{enumerate}
\end{lem}

\noindent In general  $\dim'$ might not be an ordinal dimension on $T'$, as $\dim'(X')$ might be different from $\dim'( X'( \sN'))$ where $\sN'$ is an elementary extension of $\sM'$. When $T$ defines Morley rank, we can easily check that $\dim'$ is an ordinal dimension on  $T'$, which we will refer to as the {\bf induced dimension} on $T'$.

 \medskip \noindent We say  {\bf $T$ defines multiplicity} (also known as having the DMP in the literature) if for all $L$-formulas $\varphi(x,y)$, ordinals $\alpha$, and $n$, there is an $L$-formula $\mu_{\alpha,n}(y)$ such that for all $\sM \models T$ and $b \in M^y$ we have that 
$$ \sM \models \mu_{\alpha,n}(b) \text{ if and only if } \dim \varphi(\sM,b) = \alpha  \text{ and }  \mult\  \varphi(\sM, b) =n. $$
In particular, if $T$ defines multiplicity, then $T$ defines Morley rank, and the induced dimension on $T'$ is well-defined.

\begin{prop} \label{prop: Defininginduceddimension}
Suppose $T$ defines multiplicity. Then $T'$ defines pseudo-denseness over $T$ if and only if  $T'$ defines induced dimension.
\end{prop}

\begin{proof}
Suppose $T'$ defines pseudo-denseness and $\varphi'(x,y)$ is an $L'$-formula.
Let  $(X'_{b})_{b\in M^y}$ be the family of subsets of $M^x$ defined by $\varphi'(x,y)$. 
Using the assumption that $T'$ defines pseudo-denseness and a standard compactness argument, we obtain a family $(X_{c})_{c\in M^z}$ defined by a formula whose choice might depend on $\varphi'(x,y)$ but not on $\sM'$, such that for every $b \in M^y$, $X'_b$ has a pseudo-closure that is a member of the family $(X_{c})_{c\in M^z}$. It follows from Proposition~\ref{prop:induced-1} that $\dim'(X'_{b}) = \alpha$ for $b \in M^y$ if and only there is $c \in M^z$ such that $X'_b$ is pseudo-dense in $X_c$ and $\dim(X_c)= \alpha$. As $T$ defines Morley rank and $T'$ defines pseudo-denseness, it follows that $T'$ defines induced dimension.

Now suppose $T'$ defines induced dimension. Let $\sC$ be the collection of almost irreducible subsets of $T$-models. Then $\sC$ is a collection of pseudo-cells for $T$. As $T$ defines multiplicity, $T$ defines $\sC$-membership.
 So by Proposition \ref{Prop: Definingpdoverpc}, it suffices to show $T'$ defines pseudo-denseness over $\sC$.  Let $(X'_{b})_{b\in M^y}$  and $(X_{c})_{c\in M^z}$  be a families defined by an $L'$-formula  $\varphi'(x,y)$ and an $L$-formula $\varphi(x,z)$. It follows from Proposition \ref{prop:induced-1} that when $X_c $ is in $\sC$, $X'_{b}$ is pseudo-dense in $X_c$ if and only if $\dim'(X \cap  X') = \dim(X)$. The desired conclusion follows.
\end{proof}

\begin{rem}
If $T$ defines Morley rank, then $\mult'$ is preserved under elementary extensions, so we may speak of induced multiplicity on $T'$. 
There is also an analogue of Proposition~\ref{prop: Defininginduceddimension} that involves both $\dim'$ and $\mult'$: If $T$ defines Morley rank, then $T'$  defines pseudo-denseness if and only if $T'$ defines induced dimension and induced multiplicity.
We do not include it here as we do not have an application in mind.
\end{rem}

\noindent We get the main result of this section, which is a restatement of Theorem~\ref{thm:3}: 

\begin{thm}\label{thm:psd2}
Suppose $T_\cap$ is $\aleph_0$-stable and defines multiplicity. 
If each $T_i$ defines induced dimension, then $T^{*}_\cup$ exists.
\end{thm}

\begin{proof}
This is an immediate consequence of Theorem~\ref{thm:approxinterp},  Proposition~\ref{prop:induced-1}, and Proposition~\ref{prop: Defininginduceddimension}.
\end{proof}

\begin{rem} Proposition~\ref{prop: Defininginduceddimension} and Theorem~\ref{thm:psd2} are mainly of interest because there are several situations where the natural dimension is induced dimension. Proposition~\ref{prop:induceddimandeliminateexisitinfty} below presents a general class of such situations.
In forthcoming work of the third author and Aschenbrenner it will be shown that $T'$ defines induced dimension when $T$ is $\mathrm{ACF}_0$ and $T'$ is either $\mathrm{DCF}_0$ or the theory of $(\mathbb{C};\mathbb{R})$ or $(\mathbb{C}_p;\mathbb{Q}_p)$.
\end{rem}

 \noindent The {\bf algebraic dimension} $\mathrm{adim}(X)$ of an $\sM$-definable set $X$ is the maximal $k$ for which there is $a = (a_1, \ldots, a_k)$ in the extension $X(\monster)$ of $X$ to an $\aleph_0$-saturated model $\monster$ such that (after permuting coordinates) $a_1,\ldots,a_k$ are $\text{acl}$-independent over $M$. 
It is well-known that algebraic dimension is an ordinal dimension on $\Th(\sM)$, which coincides with Morley rank for strongly minimal theories.
Fact~\ref{fact:acldim} is also well known (see~\cite[Lemma 2.2]{Cha-Pi}). 

\begin{fact} \label{fact:acldim}
A theory defines algebraic dimension if and only if it eliminates $\exists^\infty$.
\end{fact}

\begin{prop} \label{prop:induceddimandeliminateexisitinfty}
Suppose $T$ is strongly minimal and $\acl'$ agrees with $\acl$ in all $T'$-models. Then $T'$ defines induced dimension if and only if $T'$ eliminates $\exists^\infty$.
\end{prop}

\begin{proof}
Suppose $\sM' \models T'$, and 
$\sM = \sM' | L$. Since $T$ is strongly minimal, $\dim = \text{adim}$. We write $\dim'$ for the induced dimension on $T'$ and $\text{adim}'$ for the algebraic dimension in  $\sM'$. Using Fact \ref{fact:acldim},  both directions of the equivalence will be proved if we show that $\dim' = \text{adim}'$. 

If $X'$ is an arbitrary $\sM'$-definable subset of $M^x$,
$$ \dim'(X') = \min \{ \text{adim}( X ) \mid X\subseteq M^x \text{ is } \sM\text{-definable, and }  X' \subseteq X\}. $$
As $\acl' = \acl$, whenever $a \in X'(\monster') $ has $k$ components that are $\acl'$-independent over $M$, these components are also $\acl$-independent over $M$, and we have $a\in X(\monster')$ for any $\sM$-definable $X$ such that $X'\subseteq X$. 
Hence, $\text{adim}'(X') \leq \dim'(X')$.  

Conversely, let $X \subseteq M^x$ be a pseudo-closure of $X'$, and $k = \text{adim}(X)$. Then $X'$ is not contained in any $\sM$-definable set of smaller dimension. Since the set of $\sM$-definable sets of dimension less than $k$ is closed under finite unions, by compactness there is some $a'\in X'(\monster')$ that is not contained in any $\sM$-definable set of dimension less than $k$. If $a'$ does not have $k$ components that are $\acl'$-independent over $M$, then since $\acl' = \acl$, this dependence is witnessed by $a'\in Y$, where $Y$ is $\sM$-definable and $\text{adim}(Y) < k$. This contradicts the choice of $a'$. 
\end{proof}

\noindent When $T$ is the theory of algebraically closed fields and $T'$ is the theory of algebraically closed valued fields, $T'$ eliminates $\exists^\infty$ and $\acl'$ agrees with $\acl$ in all $T'$-models; see \cite{lou-dimension} for details. Thus, using Proposition~\ref{prop:induceddimandeliminateexisitinfty} and Theorem~\ref{thm:psd2}, we obtain a new proof of the existence of a model companion for the theory of algebraically closed fields with multiple valuations.

\section{$\aleph_0$-categorical and $\aleph_0$-stable base}\label{sec:stable-categorical}
\noindent In this section, we generalize Winkler's result~\cite{Winkler} on model companions of disjoint unions of theories to allow $T_\cap$ to be any complete $\aleph_0$-stable and $\aleph_0$-categorical theory with weak elimination of imaginaries. Our generalization applies to several other examples from Section~\ref{sec:examples}, such as the random graph and other combinatorial Fra\"iss\'e limits, and generic Skolemizations.  

\medskip\noindent Throughout, we keep the additional notational conventions described at the beginning of Section~\ref{sec: Pseudo-topological base} and further assume that $T$ is $\aleph_0$-categorical with only infinite models and $\aleph_0$-stable. We write $\dim$ for Morley rank on $T$ and $\mult$ for Morley degree on $T$. 

\medskip \noindent  The $\aleph_0$-stable assumption allows us to make extensive use of Proposition~\ref{prop:induced-1}, which ensures that every subset of a model of $T$ has a pseudo-closure.
It also provides us  with the following ``inductive'' procedure to check whether a subset is pseudo-dense in an almost irreducible set. Let $\sM$ be a model of $T$. Recall from Section~\ref{section:aleph-base} that $\subseteq_\text{a}$ denotes the almost subset relation, and $=_{\text{a}}$ denotes the almost equality relation.
A collection  $\sD$ of almost irreducible subsets of $M^x$ in $\sM \models T$ is {\bf representative} if $\sD$ contains a (not necessarily unique) representative for each almost equality class. 

\begin{lem}\label{lem:wink-1} 
Suppose $X \subseteq M^x$ is almost irreducible, $\sD$ is a  representative collection of almost irreducible subsets of $M^x$, and $A$ is a subset of $M^x$.
For $\alpha < \dim X$, let  $\sD_\alpha(A, X) \subseteq \sD$  consist of all $Y \in \sD$ such that:
$$ \dim Y = \alpha,\  A \text{ is pseudo-dense in } Y, \text{ and }   Y \subseteq_{\text{a}} X. $$
If $\sD_\beta(A, X) = \emptyset$ for all $\alpha< \beta < \dim X$, then:
\begin{enumerate}
        \item If  $\sD_\alpha(A, X)$ is infinite up to almost equality, then $A$ is pseudo-dense in $X$.
        \item If $\sD_\alpha(A,X)$ is finite up to almost equality, $X_\alpha^1, \ldots, X_\alpha^n$ are representatives of the almost equality classes, and
         $$ A' = A \setminus \bigcup_{i=1}^n X_\alpha^i,$$
         then $\sD_\beta(A', X) = \emptyset$ for all  $\alpha\leq  \beta < \dim X$, and $A$ is pseudo-dense in $X$ if and only if $A'$ is. 
\end{enumerate}
\end{lem}

\begin{proof}
As $\sM$ is $\aleph_0$-stable, $A \cap X $ has a pseudo-closure $Y$ that is a subset of $X$ by Proposition~\ref{prop:induced-1}. Suppose $\sD_\beta(A, X) = \emptyset$ for all $\alpha< \beta < \dim X$. Then either $\dim Y \leq \alpha$ or $\dim Y = \dim X$. If $\sD_\alpha(A, X)$ is infinite up to almost equality, then $\dim Y > \alpha$, and so $\dim Y= \dim X$. The latter implies $A$ is pseudo-dense in $X$ by Proposition \ref{prop:induced-1}. Thus we get statement (1). 

Now suppose $X_\alpha^1, \ldots, X_\alpha^n$ and $A'$ are as stated in (2). Since $A'$ is a subset of $A$, $\sD_\beta(A', X)$ is a subset of $\sD_\beta(A, X)$ for all $\beta$. 
So in particular, $\sD_\beta(A', X) = \emptyset$ for all $\alpha< \beta < \dim X$.  Suppose $X_\alpha$ is an element of $\sD_\alpha(A', X)$. 
Then $A$ is also pseudo-dense in $X_\alpha$ and so $X_\alpha =_{\text{a}} X_\alpha^i$ for some $i \in \{1, \ldots, n\}$. 
As $X_\alpha^i \cap A' =\emptyset$, $X_\alpha^i$ and  $X_\alpha$ are both almost irreducible, and $\dim X_\alpha^i = \dim X_\alpha$, it follows from Lemma \ref{lem: basicfacts} that $A'$ is not pseudo-dense in $X_\alpha$, which is absurd. 
Thus, $$\sD_\beta(A', X)=\emptyset \quad \text{for all } \alpha\leq \beta < \dim X.$$ If $A'$ is pseudo-dense in $X$, then clearly $A$ is. Suppose $A'$ is not pseudo-dense in $X$. Then $A'\cap X$ has a pseudo-closure $Y'$ with $\dim Y'< \dim X$. It follows that $A$ has a pseudo-closure $Y$ that is a subset of $Y'\cup X_\alpha^1\cup \ldots \cup X_\alpha^n$. Then \[\dim Y \leq \max(\dim Y',\alpha) < \dim X,\] and so $A$ is not pseudo-dense in $X$. We have thus obtained all the desired conclusions in (2).
\end{proof}

\noindent Lemma~\ref{lem:wink-1} is hardly useful if the purpose is defining pseudo-denseness for a general $\aleph_0$-stable theory. 
The issue is that many of the objects involved in the previous lemma are not definable.
However, many of them are definable under the additional assumption of $\aleph_0$-categoricity. 
We recall a number of facts about  $\aleph_0$-categorical and $\aleph_0$-stable theories.

\begin{fact} \label{fact:aleph0categoricalaleph0stable}
The first two statements below only require $\aleph_0$-categoricity:
\begin{enumerate}
    \item If $T$ has only infinite models, then $T$ is complete.
    \item For all finite $x$, there are finitely many formula $\varphi(x)$ up to $T$-equivalence.
    \item $T$ defines multiplicity.
      \item (\cite{CHL}, Theorem 5.1)  $\sM$ has finite Morley rank. That is, for all finite $x$, $\dim M^x < \omega$.
      \item (\cite{CHL}, Theorem 6.3)  If $x$ is a finite tuple of variables, and $p \in S^x( \sM)$, then $p$ is definable over an element of  $M^x \times M^x$. 
\end{enumerate}
\end{fact}

\noindent We now prove a key lemma that does not hold outside of the $\aleph_0$-categorical setting.  An $L$-formula $\psi(x,z)$ is {\bf representative} for the tuple of variables $x$ if it defines in every $\sM \models T$ a representative collection of almost irreducible sets. 

\begin{lem} \label{Lemma: representativeformula}
    For each finite tuple of variables $x$,  there is a representative formula $\psi(x,z)$.
\end{lem}
 
\begin{proof}

Fix $\sM \models T$ and a finite tuple $x$ of variables.  We claim that every almost irreducible subset $X$ of $M^x$ is almost equal to a subset $X'$ of $M^{x}$ such that  $X'$ is definable over an element of $M^w = M^x \times M^x$. 

Let $p \in S^x(M)$ be the generic type of $X$ and $p^\mathrm{eq}$ the unique element of $S^x(M^\mathrm{eq})$ extending $p$.  Fact \ref{fact:aleph0categoricalaleph0stable}(5), gives us $d \in M^w = M^x \times M^x$ such that $p$ is definable over $d$. Then $p^\mathrm{eq}$ is definable over $d$ and therefore stationary over $\acl^\mathrm{eq}(d)$. Hence, 
$$q = p^\mathrm{eq} | S^x(\acl^\mathrm{eq}(d)) \text{ has } \mult(q) = 1.$$ 
Let $X'' \subseteq M^x$ be defined by a formula in $q$ such that $X''$ has minimal Morley rank and degree. Then $X''$ is $\sM^\mathrm{eq}$-definable over 
$\acl^\mathrm{eq}(d)$
and $X'' =_{\text{a}}X$.
Let $X''_1, \ldots, X''_k$ are the finitely many conjugates of $X''$ by $\text{Aut}(\sM^\mathrm{eq}\slash d)$, and set $X' = \bigcap_{i =1}^k X''_i$. It is easy to check that $X'$ is definable over $d$ and $X'=_{\text{a}} X$. This completes the proof of the claim.

By Fact \ref{fact:aleph0categoricalaleph0stable}(2) there are finitely many formulas  $\psi_1(x,w), \ldots, \psi_l(x,w)$ such that every $L$-formula in variables $(x,w)$ is $T$-equivalent to one of these. Through routine manipulation, we can get a finite tuple $z$ of variables and a formula $\psi(x,z)$ and such that for all $i \in I$ and $d \in M^w$, there is  $c \in M^z$ with $\psi_i(\sM, d) = \psi(\sM,c)$. Using the claim, if  $(X_c)_{c \in M^z}$  is the family of subsets of $M^x$ defined by $\psi(x,z)$, then every almost irreducible $X$ is almost equal to $X_c$ for some $c \in M^z$. 

Finally, by Fact~\ref{fact:aleph0categoricalaleph0stable}(3), $T$ defines multiplicity. So we can modify $\psi(x,z)$ to ensure that every member of the family  $(X_c)_{c \in M^z}$ is almost irreducible. The result is a representative formula for $x$ because by Fact~\ref{fact:aleph0categoricalaleph0stable}(1), $T$ is complete.
\end{proof}

 \noindent A {\bf function up-to-permutation} from $Z \subseteq M^z$ to $M^w$ is a relation $f \subseteq Z \times M^w$ satisfying the following two conditions:
\begin{enumerate}
    \item For all $c \in Z$, there is $d \in M^w$ such that $(c,d) \in f$. 
    \item If $(c,d)$ and  $(c,d')$ are both in $f$, then $d$ is a permutation of $d'$.
\end{enumerate}
A function up-to-permutation $f$ determines an ordinary function  $\tilde{f}: Z \to M^w \slash \sim$, where $\sim$ is the equivalence relation defined by permutations. We are interested in $f$ instead of $\tilde{f}$, as it is possible that $f$ is $\sM$-definable while $\tilde{f}$ is only $\sM^\eq$-definable. For $C \subseteq Z$, we will write $f(Z)$ for the set
$$ \{ d \in M^w \mid  \text{ there is } c \in C \text{ such that } (c,d) \in f \}. $$ It is easy to observe that $|\tilde{f}(Z)|\leq |f(Z)| \leq |w|! |\tilde{f}(Z)|$ with $\tilde{f}$ as above. In particular, $f(Z)$ is finite if and only if $\tilde{f}(Z)$ is. 

\medskip \noindent Recall that $T$ \textbf{weakly eliminates imaginaries} if for all $\sM\models T$ and all $b\in \sM^{\text{eq}}$, there exists $a\in \sM$ such that $a\in \acl^{\text{eq}}(b)$ and $b\in \dcl^{\text{eq}}(a)$. The following Fact~\ref{fact: weakeliminationofimaginary} only uses the assumption that $T$ is complete.

\begin{fact} \label{fact: weakeliminationofimaginary}  Suppose $T$ weakly eliminates imaginaries.
For all $\sM \models T$, 0-definable $Z \subseteq M^z$, and 0-definable equivalence relation $R \subseteq Z^2$, there is a tuple  $w$ of variables  and a $0$-definable function up-to-permutation from $Z$ to $M^w$ such that $c R c'$ in $Z$ if and only if $f(c) = f(c')$. Moreover, the choice of formula defining $f$ can be made depending only on the choices of $L$-formulas defining $Z$ and $R$ but not on the choice of $\sM$.
\end{fact}

Below is the key proposition for this section. The main geometrical idea of the proof is already contained in Lemma~\ref{lem:wink-1}, we just need to check that everything can be carried out definably.

\begin{prop}\label{prop:wink-2} Suppose $T$ weakly eliminates imaginaries.
Then $T'$ defines pseudo-denseness over $T$ if and only if $T'$ eliminates $\exists^\infty$.
\end{prop}

\begin{proof} 
Let $\sM' \models T'$ and $\sM = \sM' | L$. In this proof, everything will be uniform in $\sM'$: when we say ``definable'', we mean ``definable by a formula which does not depend on the choice of $\sM'$''. Alternatively, we can obtain this uniformity for free by working in a sufficiently saturated model. 

For the forward direction, suppose  $T'$ defines pseudo-denseness. Let $(X'_b)_{b \in M^y}$ be a family of subsets of $M^x$ defined by an $L'$-formula $\varphi'(x,y)$. Our job is to show that the set
$$   \{ b \in M^y: X'_b \text{ is infinite} \}$$ 
is definable by an $L'$-formula.

Using Lemma~\ref{Lemma: representativeformula}, we get a representative formula $\psi(x,z)$ for the tuple of variables $x$. Let  $(X_c)_{c \in M^z}$ be the family defined by $\psi(x,z)$. Note that $X'_b$ is pseudo-dense in each of the almost irreducible components of its pseudo-closure. So $X'_b$ is infinite if and only if there is $c \in M^z$ such that
$$ X'_b \text{ is pseudo-dense in } X_c \text{ and } \dim X_c>0. $$
As $T'$ defines pseudo-denseness, the set of  pairs $(b,c)$ with $X'_b$ pseudo-dense in $X_c$ is definable by an $L'$-formula. 
By Fact \ref{fact:aleph0categoricalaleph0stable}, $T$ defines Morley rank, so the set of $c \in M^z$ with $\dim X_c>0$ is definable by an $L$-formula. Hence, we get the desired conclusion. 

For the backward implication, suppose $T'$ eliminates $\exists^\infty$. Let $(X'_b)_{b \in M^y}$ and $(X_c)_{c \in M^z}$ be families of subsets of $M^x$ defined by an $L'$-formula $\varphi'(x,y)$ and and $L$-formula $\psi(x,z)$, respectively. Set
$$  P=  \{(b,c) \in M^{(y,z)} \mid X'_b \text{ is pseudo-dense in } X_c \}.   $$
We need to show that $P$ is definable by an $L'$-formula.

We first reduce to the special case where  $\psi(x,z)$ is a representative formula for the tuple of variables $x$. Using Lemma~\ref{Lemma: representativeformula}, we get a representative formula $\delta(x,w)$ for the tuple of variables $x$. Let $(X_d)_{d \in M^w}$ be the family of subsets of $M^x$ defined by $\delta(x,w)$.  For  $b \in M^y$ and $c \in M^z$, $(b, c) \in P$  if and only if $X'_b$ is pseudo-dense in $X_d$ for all $d \in M^w$ with $X_d\subseteq_{\text{a}} X_c$ and $\dim X_d = \dim X_c$. Again,  $T$ defines Morley rank, so $T$ defines the relation of being almost a subset.  Hence, we can deduce the general case from this special case.

We decompose $P$ into finitely many sets, which we will then show to be definable using induction. For $\gamma \leq  \dim M^x$, set
$$ P^\gamma= \{ (b,c) \in P\mid  \dim X_c =\gamma\}.$$
Then $P= \bigcup_{\gamma\leq \dim M^x} P^\gamma$. Therefore, as $\dim(M^x)$ is finite, it suffices to show that for all $\gamma$, $P^\gamma$ is definable. 

We will proceed by induction on $\gamma$, uniformly with respect to the choice of the formula $\varphi(x,y)$. For $\gamma =0$, by Lemma~\ref{lem: basicfacts}(1), $X'_b$ is pseudo-dense in a finite set if and only $X'_b$ contains that finite set. So $(b,c) \in M^{(y,z)}$ is in $P^0$ if and only if $\dim(X_c) =0$ and $X_c \subseteq X'_b$. Hence, $P^0$ is definable by an $L'$-formula.

Now, assume $\gamma>0$ and we have proven the statement for all smaller values of $\gamma$. Let $\sD = (X_c)_{c \in M^z}$ be the representative collection of almost irreducible subsets of $M^x$ defined by $\psi(x,z)$. Toward applying Lemma \ref{lem:wink-1},  set
$$D_{\alpha, b, c} = \{ d \in M^z \mid (b,d) \in P^\alpha \text{ and } X_d \subseteq_{\text{a}} X_c \} $$
for each ordinal $\alpha$ and each $(b,c) \in M^y \times M^z$.
In other words, if $\sD_\alpha(X'_b, X_c)$ is defined as in Lemma \ref{lem:wink-1}, then $d \text{ is in } D_{\alpha, b, c}$ if and only if $X_d  \text{ is in }\sD_\alpha(X'_b, X_c).$ For all $\alpha<\gamma$, since $T$ defines Morley rank and $P^{\alpha}$ is definable by the inductive hypothesis, the family $(D_{\alpha, b, c})_{(b,c) \in M^{(y,z)}}$ is definable by an $L'$-formula.

For $\alpha< \gamma $, set
$$ P^\gamma_\alpha = \{  (b,c) \in P^\gamma \mid D_{\beta, b, c} =\emptyset \text{ for all } \alpha <  \beta < \gamma \}. $$
Then $P^\gamma = P^\gamma_{\gamma-1}$. Hence, we can get the desired conclusion by showing the stronger fact that $P^\gamma_\alpha$ is definable by an $L'$-formula for all $\alpha< \gamma \leq  \dim M^x$.

Now we proceed by an inner induction on $\alpha$.
When $\alpha = 0$, we get from Lemma \ref{lem:wink-1} that $(b,c) \in M^{(y,z)}$ is in $P^\gamma_0$ if and only if $$\dim X_c = \gamma,\quad D_{\beta, b, c} = \emptyset \text{ for all }0< \beta< \gamma, \quad \text{ and } X'_b \text{ is infinite. } $$
Hence, $P^\gamma_0$ can be defined by an $L'$-formula, by the assumption that $T'$ eliminates $\exists^\infty$ and  the fact that $T$ defines Morley rank. 

Suppose $0< \alpha< \gamma$ and we have shown the statement for all smaller values of $\alpha$. As noted above, for all $\beta<\gamma$, 
the families  $( D_{\beta, b, c})_{(b,c)\in M^y \times M^z}$  are definable by $L'$-formulas. Recall that $T$ defines Morley rank and weakly eliminates imaginaries. By Fact \ref{fact: weakeliminationofimaginary},  there is $w$ and a $L$-definable function up-to-permutation $f$ from $Z$ to $M^w$, such that for all $d_1$ and $d_2$ in $M^z$,
$$ f(d_1) = f(d_2)   \text{ if and only if } X_{d_1} =_{\text{a}} X_{d_2}.$$
In particular, the family $( f (D_{\alpha,b,c}))_{(b,c) \in Y \times Z_\gamma}$ can be defined by an $L'$-formula. As $T'$ eliminates $\exists^\infty$, there is $n$ such that  
$$|f (D_{\alpha,b,c})| > n|w|! \text{ implies } f (D_{\alpha,b,c}) \text{ is infinite}.$$  
Now let $\tilde{Y}$  be the set of $\tilde{b}=(b,c, d_1, \ldots, d_n)$ in $M^y \times M^z \times M^z\times \ldots \times M^z$ such that the following properties hold:
\begin{enumerate}
    \item $\dim X_c = \gamma$ and $D_{\beta, b, c} = \emptyset$ for all $\alpha< \beta< \gamma$.
    \item  $f (D_{\alpha,b,c})$ is finite.
    \item   $\dim X_{d_i} = \alpha $ and  $X'_b$ is pseudo-dense in $X_{d_i}$ for $i \in \{1, \ldots, n\}$.
    \item If $\dim X_{d} = \alpha$ and $X'_b$ is pseudo-dense in $X_d$ for some $d \in M^z$, then $X_d=_{\text{a}}X_{d_i}$ for some $i \in \{1, \ldots, n\}$.
\end{enumerate}
By the inductive hypothesis and the fact that $T$ defines Morley rank, $\tilde{Y}$ is definable by an $L'$-formula. For each $\tilde{b} \in \tilde{Y}$, set
$$ \tilde{X}'_{\tilde{b}} =    X'_b \setminus \bigcup_{i=1}^n X_{d_i}.$$
For $\tilde{b} \in M^{\tilde{y}} \setminus \tilde{Y}$, let $\tilde{X}'_{\tilde{b}}$  be the empty set.
Then, the family $( \tilde{X}'_{\tilde{b}} )_{\tilde{b} \in M^{\tilde{y}}}$ is definable by an $L'$-formula $\tilde{\varphi}'(x, \tilde{y})$. We obtain  $\tilde{P}^{\gamma}_{\alpha-1}$  from $\tilde{\varphi}'(x,\tilde{y})$ and $\psi(x,z)$ in the same fashion as we obtain $P^\gamma_{\alpha-1}$ from $\varphi'(x,y)$ and $\psi(x,z)$.
The induction hypothesis, applied to the formula $\tilde{\varphi}'(x, \tilde{y})$, 
implies that $\tilde{P}^{\gamma}_{\alpha-1}$ is definable. From Lemma \ref{lem:wink-1}, 
$(b,c)$ is in  $P^\gamma_{\alpha}$ if and only if $\dim X_c = \gamma$ and $D_{\beta, b, c} = \emptyset$ for all $\alpha< \beta< \gamma$ and either of the following hold:
\begin{enumerate}
    \item  $f (D_{\alpha,b,c})$ is infinite.
    \item  There are $d_1, \ldots, d_n$ in $M^z$ with $\tilde{b}=(b,c, d_1, \ldots, d_n)\in\tilde{Y}$ and 
    $ (\tilde{b},c) \in \tilde{P}^{\gamma}_{\alpha-1}$.
\end{enumerate}
Thus $P^\gamma_{\alpha}$ is definable, which completes the proof.
\end{proof}

\noindent We get the main result of this section, which is a restatement of Theorem~\ref{thm:2}:

\begin{thm}\label{thm:psd1}
Suppose $T_\cap$ is complete, $\aleph_0$-stable, and $\aleph_0$-categorial.
If $T_\cap$ weakly eliminates imaginaries, and each $T_i$ eliminates $\exists^\infty$, then $T^{*}_\cup$ exists. 
If $T_i^{\eq}$ eliminates $\exists^\infty$ for all $i$, then $T^{*}_\cup$ exists.
\end{thm}

\begin{proof}
The first statement follows from Theorem~\ref{thm:approxinterp}, Proposition~\ref{prop:induced-1}, and Proposition~\ref{prop:wink-2}. 
The second statement then follows from the first statement and Remark~\ref{rem: robust}(4): if $T_i^{\eq}$ eliminates $\exists^\infty$, so does $T_i^{\cap-\eq}$, so we may assume $T_\cap$ eliminates imaginaries.
\end{proof}

\noindent The conditions of Theorem~\ref{thm:psd1} are satisfied when $L_\cap$ is the empty one-sorted language and $T_\cap$ is the theory of infinite sets.
So we recover Winkler's theorem on disjoint unions of theories~\cite{Winkler}. 
Using similar ideas, we recover Winkler's theorem on generic Skolemizations; see Section~\ref{sec:examples}.

\medskip \noindent The assumptions of Corollary~\ref{cor:trivial-theory} are very strong, but it is applicable more often then one might expect.
For example, it applies to the random graph and many other combinatorial Fra\"iss\'e limits, as presented in Section~\ref{sec:examples}.

\begin{cor}
\label{cor:trivial-theory}
Suppose $T_\cap$ is complete and each $T_i$ is interpretable in the theory of infinite sets.
Then $T^*_\cup$ exists.
\end{cor}

\begin{proof}
The theory of infinite sets is $\aleph_0$-stable, $\aleph_0$-categorical,  and eliminates $\exists^\infty$ in $T^{\eq}$.
Each of these three properties is preserved under interpretations, so if $T_i$ is interpretable in the theory of infinite sets then $T$ is $\aleph_0$-stable, $\aleph_0$-categorical, and $T_i^{\eq}$ eliminates $\exists^\infty$.
So the result follows from Theorem~\ref{thm:psd1}.
\end{proof}

\noindent The theory $T_q$ of vector spaces over a finite field with $q$ elements is $\aleph_0$-stable, $\aleph_0$-categorical, and weakly eliminates imaginaries.
Thus any theory $T'$ extending $T_q$ defines pseudo-denseness if and only if it eliminates $\exists^\infty$.
This does not generalize to vector spaces over characteristic zero fields, which are $\aleph_0$-stable and weakly eliminate imaginaries, but are not $\aleph_0$-categorical.
For example, let $T$ be the theory of torsion-free divisible abelian groups (vector spaces over $\mathbb{Q}$).
Let $T'$ be $\ACF_0$, note that $T'$ is an expansion of $T$.
Then $T'$ does not define pseudo-denseness over $T$.
Suppose $\monster'$ is an $\aleph_1$-saturated model of $T'$.
Let
$$ L = \{ (a,b,c) \in \mathbf{M}^3 \mid ab = c \} $$
and consider the definable family $\{ L_a \mid a \in M \}$ where $L_a = \{ (b,c) \in \mathbf{M}^2 \mid ab = c \} $.
We leave the easy verification of Lemma~\ref{lem:Q-pdense} to the reader:

\begin{lem}
\label{lem:Q-pdense}
Fix $a \in \mathbf{M}$.
Then $L_a$ is pseudo-dense in $\mathbf{M}^2$ if and only if $a \notin \mathbb{Q}$.
\end{lem}

\noindent As $\mathbb{Q}$ is countable and infinite it cannot be a definable set in an $\aleph_1$-saturated structure.
Thus $\monster'$ does not define pseudo-denseness over $(\mathbf{M};+).$
The same argument shows that any theory expanding $T'$ does not define pseudo-denseness over $T$.

\medskip \noindent There is a natural rank \text{rk} on any $\aleph_0$-categorical theory, described in \cite[Section 2.3]{Pierre} and \cite[Section 2.2.1]{ChrHru}.
This rank is known to agree with thorn rank on $\aleph_0$-categorical structures, so it is an ordinal rank on rosy $\aleph_0$-categorical theories.
A special case of Theorem~\ref{thm:psd4} is that any expansion of the theory DLO of dense linear orders without endpoints defines pseudo-denseness over DLO with respect to $\text{rk}$ (which agrees with the usual o-minimal dimension over DLO).
This fact, together with Proposition~\ref{prop:wink-2}, and recent groundbreaking work on NIP $\aleph_0$-categorical structures \cite{Pierre, Pierre2} motivates Question~\ref{ques:NIPomega}.

\begin{question}\label{ques:NIPomega}
Suppose $T$ is NIP, $\aleph_0$-categorical, and rosy.
If $T'$ eliminates $\exists^\infty$, then must $T'$ define pseudo-denseness over $T$ (with respect to $\mathrm{rk}$)?
\end{question}

\noindent Unfortunately, $\mathrm{rk}$ does not necessarily agree with Morley rank on $\aleph_0$-stable and $\aleph_0$-categorical theories.
One might hope that an approach to Question~\ref{ques:NIPomega} would synthesize the ideas of Section~\ref{section-tame-topology} and Section~\ref{sec:stable-categorical}.

\bibliographystyle{amsalpha}
\bibliography{the}

\end{document}